\documentclass[]{amsart}

\allowdisplaybreaks[1]

\usepackage{amssymb}           
\usepackage{bm}                

\usepackage{amsmath}           
\usepackage{amsthm}            
\usepackage{amscd}             

\usepackage{ascmac}            
\usepackage{indentfirst}       

\usepackage[pdftex]{graphicx} 
\newtheorem{theorem}{Theorem}[section]
\newtheorem{prop}[theorem]{Proposition}

\newtheorem{lem}[theorem]{Lemma}

\theoremstyle{definition}

\newtheorem{remark}[theorem]{Remark}
\newcommand{\R}{\ensuremath{\mathbb{R}}}

\newcommand{\Z}{\ensuremath{\mathbb{Z}}}

\newcommand{\os}{\mathbin{\overline{\ast}}}
\newcommand{\us}{\mathbin{\underline{\ast}}}
\newcommand{\s}{\mathbin{\ast}}

\newcommand{\Col}[2]{\mathrm{Col}_{#1}(#2)}
\newcommand{\ColD}[2]{\mathrm{Col}^{\Delta}_{#1}(#2)}
\newcommand{\Qdl}[1]{\mathcal{Q}(#1)}

\newcommand{\col}[1]{\mathcal{#1}}

\def\MARU#1{{\rm\ooalign{\hfil\lower.168ex
\hbox{#1}\hfil\crcr{$\bigcirc$}}}}

\makeatletter
\long\def\dow#1{\leavevmode\setbox\@tempboxa\hbox{$#1$}
\@tempdima\fboxrule
\advance\@tempdima \fboxsep \advance\@tempdima \dp\@tempboxa
\hbox{\hskip\fboxsep \lower \@tempdima\hbox
{\vbox{
\hbox{ 
\vbox{\vskip\fboxsep \box\@tempboxa\vskip\fboxsep}\hskip
\fboxsep\vrule \@width \fboxrule}%
\hrule \@height \fboxrule}}}\hskip\fboxsep }
\makeatother

\makeatletter
\long\def\up#1{\leavevmode\setbox\@tempboxa\hbox{$#1$}
\@tempdima\fboxrule
\advance\@tempdima \fboxsep \advance\@tempdima \dp\@tempboxa
\hbox{\hskip\fboxsep \lower \@tempdima\hbox
{\vbox{\hrule \@height \fboxrule 
\hbox{ 
\vbox{\vskip\fboxsep \box\@tempboxa\vskip\fboxsep}\hskip
\fboxsep\vrule \@width \fboxrule}%
}}}\hskip\fboxsep }
\makeatother

\begin{document}



\date{March 5, 2020.}


\title[Quandle colorings vs. biquandle colorings]{Quandle colorings vs. biquandle colorings}


\author{Katsumi Ishikawa}
\address{Research Institute for Mathematical Sciences, Kyoto University, 
Kitashirakawa Oiwake-cho
Sakyo-ku, Kyoto 606-8502, Japan}
\email{katsumi@kurims.kyoto-u.ac.jp}

\author{Kokoro Tanaka}
\address{Department of Mathematics, Tokyo Gakugei University, 
Nukuikita 4-1-1, 
Koganei, Tokyo 184-8501, Japan}
\email{kotanaka@u-gakugei.ac.jp}



\begin{abstract}
Biquandles are generalizations of quandles. As well as quandles, biquandles give us many invariants for oriented classical/virtual/surface links. Some invariants derived from biquandles are known to be stronger than those from quandles for virtual links. 
However, we have not found an essentially refined invariant for classical/surface links so far.  
In this paper, we give an explicit one-to-one correspondence between biquandle colorings 
and quandle colorings for classical/surface links. 
We also show that biquandle homotopy invariants and quandle homotopy invariants are equivalent. 
As a byproduct, we can interpret biquandle cocycle invariants in terms of shadow quandle cocycle invariants. 
\end{abstract}

\maketitle


\section{Introduction}\label{sec:intro}

The knot group, which is the fundamental group of the complement of 
an oriented classical knot in $\R^3$ (or an oriented surface knot in $\R^4$), 
is one of the most important knot invariants. 
For example, the number of all homomorphisms from 
the knot group to a chosen finite group is known to be a basic knot invariant.  
When we read off the group presentation, called the Wirtinger presentation, 
of the knot group, all relations are described by conjugation. 
In this view, as long as we consider oriented classical/surface knots, the conjugate operation, rather than the product, has an essential meaning. 

Here, the notion of quandles arises naturally; a quandle, introduced in \cite{joy} (and in \cite{mat} as a distributive groupoid), is a set with a binary operation satisfying three axioms, which are basic properties of the conjugate operation of a group. 
The fundamental quandle of an oriented classical link is defined, and a homomorphism, which is called a coloring, from it to another quandle is described as an assignment of elements of the quandle to the arcs of a diagram of the link; 
the number of all colorings, called the quandle coloring number,
is a useful and powerful invariant for classical links.
Some refinements, e.g., quandle cocycle invariants \cite{cjkls} and quandle homotopy invariants \cite{nos1,nos2}, of the quandle coloring number 
are defined for oriented classical/surface links, and some of them are also valid for oriented virtual links.

The axioms of quandles also correspond to the three Reidemeister moves of link diagrams. 
However, from such a diagrammatic viewpoint, leaving topological meaning, we may cut each arc at crossing points and distinguish the divided pieces: we color the semi-arcs, rather than the arcs. 
This leads to the definition of biquandles \cite{fjk}, as a generalization of quandles: a biquandle has two binary operations corresponding to going over/under a crossing.
As with quandles, biquandles bring various invariants for oriented classical/surface/virtual links. 
See \cite{fjk,car,ces} for example. 

Some invariants derived from biquandles are known to be stronger than 
those from quandles for virtual links, 
but we have not found an essentially refined invariant for classical/surface links so far.  
Related to this problem, the first author \cite{is} proved that the fundamental biquandle of an oriented classical/surface link can be recovered from its fundamental quandle.
As a corollary, a biquandle coloring number is shown to be equal to the quandle coloring number of a certain quandle (this can be seen from \cite{sol}, too), 
and other invariants derived from biquandles are also expected to be rewritten by 
those from quandles. 
However, the relation of \cite{is} is defined in a quite algebraic way, so the induced correspondence between the biquandle colorings and the quandle colorings seems complicated and it looks difficult to find out relations between other invariants.

In this paper, we give an explicit one-to-one correspondence between biquandle colorings 
and quandle colorings for classical links in Theorem~\ref{thm:main}  
(and also for surface links in Theorem~\ref{thm:main2}). 
More precisely, for any biquandle $X$ and any oriented link diagram $D$, 
we explicitly describe a bijection between the set $\Col{X}{D}$ of the $X$-colorings and $\Col{\mathcal{Q}(X)}{D}$ of the $\mathcal{Q}(X)$-colorings, where $\mathcal{Q}$ 
is a functor from the category of biquandles to that of quandles.
Using this correspondence, we interpret biquandle cocycle invariants in terms of shadow quandle cocycle invariants (Theorem \ref{thm:bci}). Furthermore, we show that biquandle homotopy invariants and quandle homotopy invariants are equivalent (Theorem \ref{thm:bhi}) when they are identified via the one-to-one correspondence of Theorem \ref{thm:main}. Although Theorem \ref{thm:bci} follows from Theorem \ref{thm:bhi}, we explicitly give a formula for cocycles in Theorem \ref{thm:bci}.

We shall explain the one-to-one correspondence roughly. 
Details will be discussed in Section~\ref{sec:corr}. 
The direction from a quandle coloring to a biquandle coloring consists of three steps. 
See Figure~\ref{fig:pull-away}. 
Given a $\mathcal{Q}(X)$-coloring on $D$, we first regard it as an $X$-coloring on the \lq\lq doubled diagram $W(D)$\rq\rq\ of $D$. Then, we ``pull away'' the backward link to obtain the diagram $D \sqcup -D^h$ and an $X$-coloring on it, 
where $-D^h$ is the horizontal mirrored diagram of $D$ with the opposite orientation.  
Finally, the coloring on $D$ of $D \sqcup -D^h$ is the required biquandle coloring on $D$. 
The other direction is very simple. 
We ``push out'' each semi-arc of the diagram with an $X$-coloring to the unbounded region and 
read off the color; this turns out to be a $\mathcal{Q}(X)$-coloring. See Figure~\ref{fig:push-out}. 
\begin{figure}[thbp]
\[\begin{CD} 
\begin{minipage}{0.15\textwidth}
\includegraphics[width=\textwidth, pagebox=artbox]{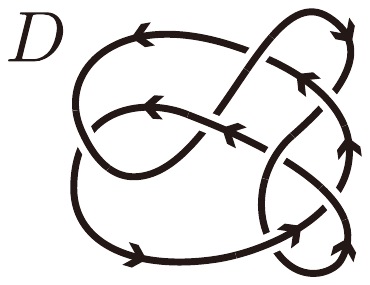} 
\\ \mbox{}
\end{minipage}\quad 
@>\mbox{\Large \lq\lq$\Psi$\rq\rq}>
\begin{minipage}{0.25\textwidth}
\mbox{}
\end{minipage}\quad 
> 
\begin{minipage}{0.15\textwidth}
\includegraphics[width=\textwidth, pagebox=artbox]{figure/original-diagram.pdf} 
\\ \mbox{}
\end{minipage}\\ 
@V\mbox{\Large \lq\lq$\Psi_1$\rq\rq}VV  
@AA\mbox{\Large \lq\lq$\Psi_3$\rq\rq}A  \\ 
\begin{minipage}{0.23\textwidth}
\includegraphics[width=\textwidth, pagebox=artbox]{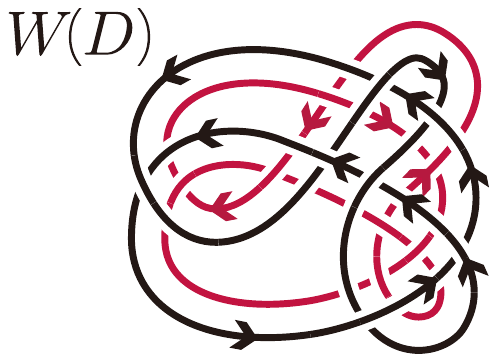} 
\\ \mbox{}
\end{minipage}\quad
@>
\begin{minipage}{0.30\textwidth}
\includegraphics[width=\textwidth, pagebox=artbox]{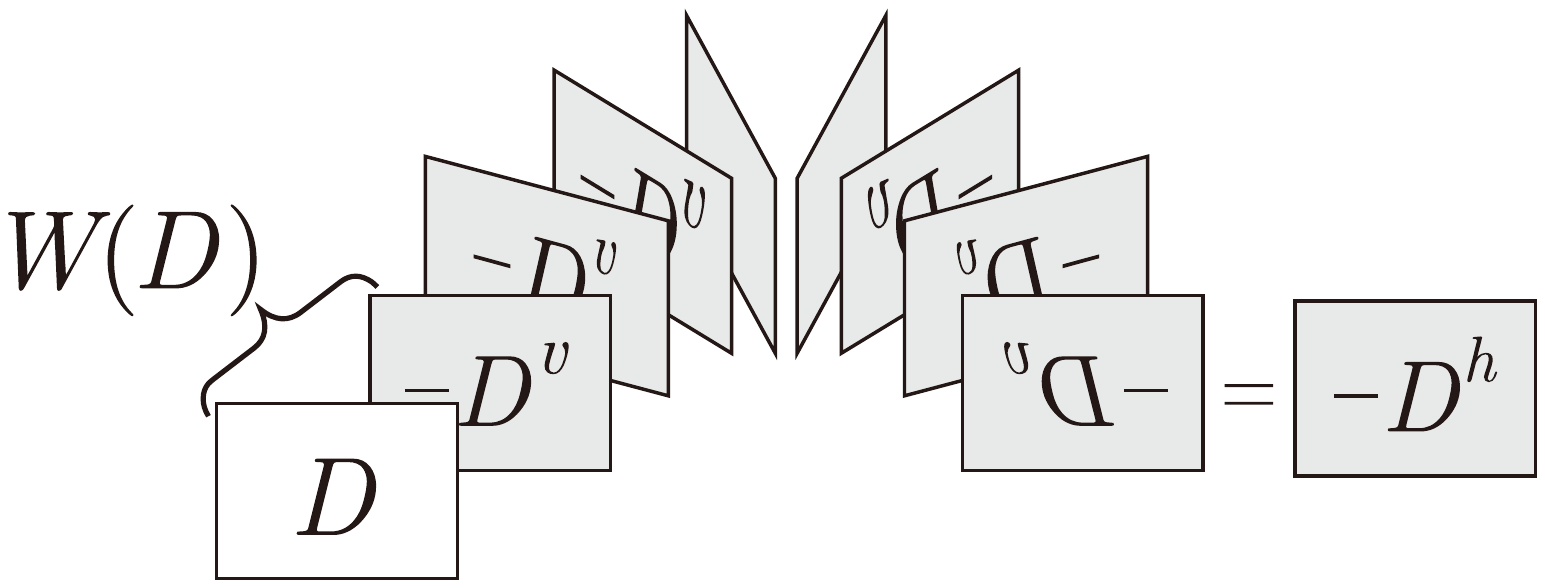} 
\end{minipage}\quad 
>\mbox{\Large \lq\lq$\Psi_2$\rq\rq}> 
\begin{minipage}{0.35\textwidth}
\includegraphics[width=\textwidth, pagebox=artbox]{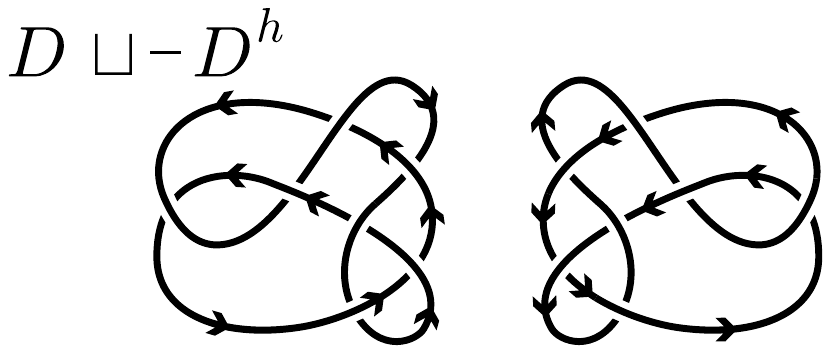} 
\\ \mbox{}
\end{minipage}
\end{CD}\]
\caption{The map $\Psi$}\label{fig:pull-away}
\end{figure}

This paper is organized as follows. In Section \ref{sec:def}, we recall basic definitions of quandles and biquandles. We then give the one-to-one correspondence between biquandle colorings and quandle colorings in Section \ref{sec:corr}. Although this definition seems to be dependent on the choice of a sequence of Reidemeister moves (see the definition of $\Psi_2$ in Section \ref{subsec:qtobq}), in fact it is independent; Section \ref{sec:alg} provides an algebraic approach to the one-to-one correspondence and especially show the independence of the map $\Psi_2$ (Lemma~\ref{lem:well-def}). 
The relation between cocycle invariants and that between homotopy invariants are discussed in Sections \ref{sec:ci} and \ref{sec:hi}, respectively. 
We give a rigorous definition of biquandle spaces (and consequently that of quandle spaces) 
in Subsection~\ref{subsec:BX} 
and point out problems on previously known definitions of quandle spaces in Remark~\ref{rem:q-space}.
Section~\ref{sec:2-dim} collects the three theorems for surface links similar to 
those (Theorems~\ref{thm:main}, \ref{thm:bci}, and \ref{thm:bhi}) for classical links. 
Appendix~\ref{sec:app} discusses a difference between the fundamental biquandle 
and the topological biquandle for oriented classical links, mentioned in Subsection~\ref{subsec:same_R-move}.    
In what follows, except in Section~\ref{sec:2-dim}, 
we call a classical link in $\R^3$ as a \textit{link} for simplicity.

\section{Definitions}\label{sec:def}
\subsection{(Bi)quandles}
A \textit{quandle} \cite{joy, mat} is a set $Q$ with a binary operation 
$\s \colon Q \times Q \rightarrow Q$ satisfying the following three axioms.
\begin{enumerate}
\item[(Q1)] For any $a \in Q$, we have $a \s a=a$. 
\item[(Q2)] For any $a \in Q$, 
the map ${}\s a \colon Q \to Q$, $\bullet \mapsto \bullet \s a$, is bijective. 
\item[(Q3)] For any $a,b,c \in Q$, we have $(a \s b) \s c=(a \s c) \s (b \s c)$.
\end{enumerate}
The axioms (Q1), (Q2), and (Q3) correspond to Reidemeister moves 
of types I, II, and III respectively. 
It follows from (Q2) that 
there exists a unique binary operation $\s^{\,-1} \colon Q \times Q \to Q$ 
satisfying 
\[
(a \s b) \s^{-1} b = (a \s^{-1} b) \s b = a 
\] 
for any $a,b \in Q$.

A \textit{biquandle} \cite{fjk} is a set $X$ with two binary operations 
$\us, \os \colon X \times X \rightarrow X$ 
satisfying the following three axioms.
\begin{enumerate}
\item[(BQ1)] For any $x \in X$, we have $x \us x = x \os x$.
\item[(BQ2)] For any $x \in X$, the two maps 
\[
{}\us x \colon X \to X,\ \bullet \mapsto \bullet \us x
\text{\quad and \quad}
\os x \colon X \to X,\ \bullet \mapsto \bullet \os x
\]
are bijective, and 
the map 
\[
H \colon X \times X \to X \times X,\ H(x,y)=(y \os x, x \us y) 
\] 
is also bijective. 
\item[(BQ3)] For any $x,y,z \in X$, we have
\[
\begin{array}{l}
(x \us y) \us (z \us y) = (x \us z) \us (y \os z), \\
(x \os y) \os (z \os y) = (x \os z) \os (y \us z), \\
(x \us y) \os (z \us y) = (x \os z) \us (y \os z).
\end{array}
\]
\end{enumerate}
The axioms (BQ1), (BQ2), and (BQ3) correspond to Reidemeister moves 
of types I, II, and III respectively.
It follows from the first half of (BQ2) that 
there exist unique binary operations $\us^{\,-1}, \os^{\,-1} \colon X \times X \to X$ 
satisfying 
\[
(x \us y) \us^{\,-1} y = (x \us^{\,-1} y) \us y = x 
\quad \text{and} \quad 
(x \os y) \os^{\,-1} y = (x \os^{\,-1} y) \os y = x 
\] 
for any $x,y \in X$.

\begin{remark}\label{rem:qbq}
For a quandle $(Q,\s)$, 
when we define a binary operation $\os \colon X \times X \to X$ by $x \os y := x$, 
the triplet $(Q, \s, \os)$ becomes a biquandle.  
Conversely, we sometimes call a biquandle $(X, \us, \os)$ a quandle if $x \os y = x$ for any $x,y \in X$.
\end{remark}

\subsection{(Bi)quandle colorings}
We review the concept of \textit{coloring}s, which are the main topic of this paper, 
of an oriented link diagram $D$ by a quandle $Q$ or a biquandle $X$. 

\par Let $\mathcal{A}(D)$ be the set of the arcs of $D$. 
A map $\mathcal{C} \colon \mathcal{A}(D) \to Q$ is a 
\textit{quandle coloring} if it satisfies the following relation at every crossing. 
Let $\alpha_i, \alpha_j, \alpha_k$ be the three arcs of $D$ around a crossing 
as in Figure~\ref{fig:crossing} (left). Then it is required that 
\[
\mathcal{C}(\alpha_i) \s \mathcal{C}(\alpha_j) = \mathcal{C}(\alpha_k) .
\]
Let $\Col{Q}{D}$ be the set of quandle colorings of the diagram $D$ 
by the quandle $Q$. 
Given a sequence of Reidemeister moves which takes a diagram $D$ to another $D'$, 
we have a bijection $\Col{Q}{D} \to \Col{Q}{D'}$. 
See \cite{fr, joy}. 
For example, the coloring number, the cardinality of the colorings, 
is an isotopy invariant of oriented links.

Divide over-arcs of $D$ at the crossings and 
call the resultant arcs \textit{semi-arc}s of $D$. 
Let $\mathcal{SA}(D)$ be the set of semi-arcs of $D$. 
We note that there is a natural surjection $\pi \colon 
\mathcal{SA}(D) \to \mathcal{A}(D)$. 
A map $\mathcal{C} \colon \mathcal{SA}(D) \to X$ is a 
\textit{biquandle coloring} if it satisfies the following relation at every crossing. 
Let $\sigma_i, \sigma_j, \sigma_k, \sigma_l$ be the four semi-arcs of $D$ around 
a crossing as in Figure~\ref{fig:crossing} (right). Then it is required that 
\[
\mathcal{C}(\sigma_i) \us \mathcal{C}(\sigma_j) = \mathcal{C}(\sigma_k) 
\text{\quad and \quad}
\mathcal{C}(\sigma_j) \os \mathcal{C}(\sigma_i) = \mathcal{C}(\sigma_l) .
\]
Let $\Col{X}{D}$ be the set of biquandle colorings of the diagram $D$ 
by the biquandle $X$. 
Given a sequence of Reidemeister moves which takes a diagram $D$ to another $D'$, we have a bijection $\Col{X}{D} \to \Col{X}{D'}$. 
See \cite{fjk}. 
For example, the coloring number, the cardinality of the colorings, 
is an isotopy invariant of oriented links.
We remark that a coloring by a quandle $Q$ can be considered as a biquandle coloring 
by composing the surjection $\pi$ and equipping $Q$ with the biquandle structure 
as in Remark~\ref{rem:qbq}. 
\begin{figure}[thbp]
\includegraphics[width=0.7\textwidth, pagebox=artbox]{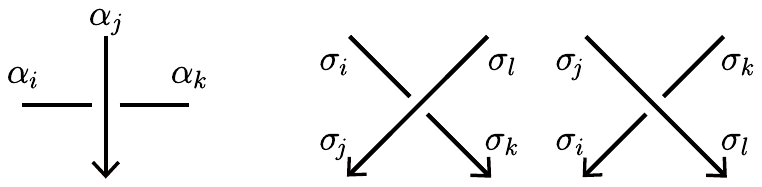}
\caption{Arcs and semi-arcs}\label{fig:crossing}
\end{figure}

\subsection{Reidemeister moves and colorings}\label{subsec:same_R-move}
We can also define a (bi)quandle coloring of a diagram of an oriented link $L$ as a homomorphism from the fundamental (bi)quandle of $L$ to the chosen (bi)quandle. 
See \cite{joy, mat} for fundamental quandles and 
\cite{fjk} for fundamental biquandles. 
Because of a topological definition of the fundamental quandle, the bijections between quandle colorings induced by ``topologically same'' sequences of Reidemeister moves are coincident: 
For two diagrams $D,D'$ of $L$, 
consider two sequences of Reidemeister moves that take $D$ to $D'$, and let $f_1,f_2$ be isotopies represented by the sequences. If $f_2^{-1} \cdot f_1$ is equivalent to the trivial motion in the sense of \cite{g}, 
then the bijections $f_1^Q,f_2^Q:\Col{Q}{D} \to \Col{Q}{D'}$ induced by the two sequences  coincide. 
In fact, this also holds for biquandle colorings:

\begin{prop}\label{prop:bqR}
Let $X$ be a biquandle. If two sequences of Reidemeister moves which take a diagram $D$ to another $D'$ are ``topologically same'', 
then the induced bijections $f_1^B,f_2^B: \Col{X}{D} \to \Col{X}{D'}$ are coincident. 
\end{prop}

\noindent 
However, we do not know a topological definition of the fundamental biquandle 
and hence this proposition is far from trivial; we give a proof in Subsection \ref{subsec:alg-psi}. 
We remark here that Horvat \cite{hor} introduced the 
\lq\lq \textit{topological biquandle} $\widehat{\mathcal{B}}_L$\rq\rq\ of an oriented link $L$. 
Although $\widehat{\mathcal{B}}_L$ can be expressed as a quotient of the fundamental biquandle 
of $L$, unfortunately, they are not isomorphic. See Appendix~\ref{sec:app}.

\section{One-to-one correspondence}\label{sec:corr}
For a biquandle $X=(X,\us,\os)$, 
we define a binary operation $\s \colon X \times X \to X $ by 
\[
x \s y := (x \us y) \os^{\,-1} y \quad (x,y \in X), 
\]
and denote the pair $(X, \s)$ by $\Qdl{X}$. 
Then it is known in \cite{ash} that $\Qdl{X}$ becomes a quandle. 
We remark here that (the solution associated with) $\Qdl{X}$ has been known 
as the \textit{derived solution}~\cite{sol} of (the solution associated with) $X$ 
in the research area of set-theoretical Yang--Baxter equations. 
See \cite{lv2} for recent development.
The following is the main result of this paper: 

\begin{theorem}\label{thm:main}
Let $X$ be a biquandle and $D$ be a diagram of an oriented link. 
Then there exists a one-to-one correspondence 
between $\Col{X}{D}$ and $\Col{\Qdl{X}}{D}$. 
\end{theorem}

To prove Theorem~\ref{thm:main}, 
we define a map $\Psi \colon \Col{\Qdl{X}}{D} \to \Col{X}{D}$ in Subsection~\ref{subsec:qtobq} and a map $\Phi \colon \Col{X}{D} \to \Col{\Qdl{X}}{D}$ in Subsection~\ref{subsec:bqtoq}, 
and check that 
one of the two maps is the inverse of the other 
in Subsection~\ref{subsec:proof}.

\begin{remark}
In Theorem \ref{thm:main}, the correspondence $\Psi \colon \Col{\Qdl{X}}{D} \to \Col{X}{D}$ is in fact natural for Reidemeister moves; see Lemma \ref{lem:compatibility} for details.
\end{remark}

\subsection{From quandle colorings to biquandle colorings}\label{subsec:qtobq}
For an oriented link diagram $D$ and a biquandle $X$, 
we give a definition of 
a map $\Psi \colon \Col{\Qdl{X}}{D} \to \Col{X}{D}$, 
which turns out to be bijective,
by composing three maps as described below. 
See also Figure~\ref{fig:pull-away} in Section~\ref{sec:intro}. 

Let us define diagrams $-D$, $D^v$, $D^h$ and $W(D)$ for $D$ as follows. 
See Figure~\ref{fig:mirror-and-double}. 
We define the inverted diagram $-D$ 
to be the same diagram as $D$ with the opposite orientation, and 
the vertical mirrored diagram $D^v$ 
to be the same diagram as $D$ with the opposite crossing information. 
The horizontal mirrored diagram $D^h$ is defined to be $m(D)$, where 
$m \colon \R^2 \to \R^2$, $(x,y) \mapsto (-x,y)$, is an involution.  
We remark that $D^h$ represents the same oriented link as the diagram $D^v$ does. 
For an oriented link diagram $D$, 
we put the diagram $-D^v$ behind $D$ 
by pushing $-D^v$ slightly into normal direction of $D$, and 
denote the resultant diagram, called the \textit{doubled diagram} of $D$, by $W(D)$. 
We remark that $W(D)$ represents the same oriented link as 
the diagram $D \sqcup -D^h$ does, 
where $D \sqcup -D^h$ is a disjoint union of $D$ and $-D^h$. 
\begin{figure}[thbp]
\begin{tabular}{c}
\includegraphics[width=0.8\textwidth, pagebox=artbox]{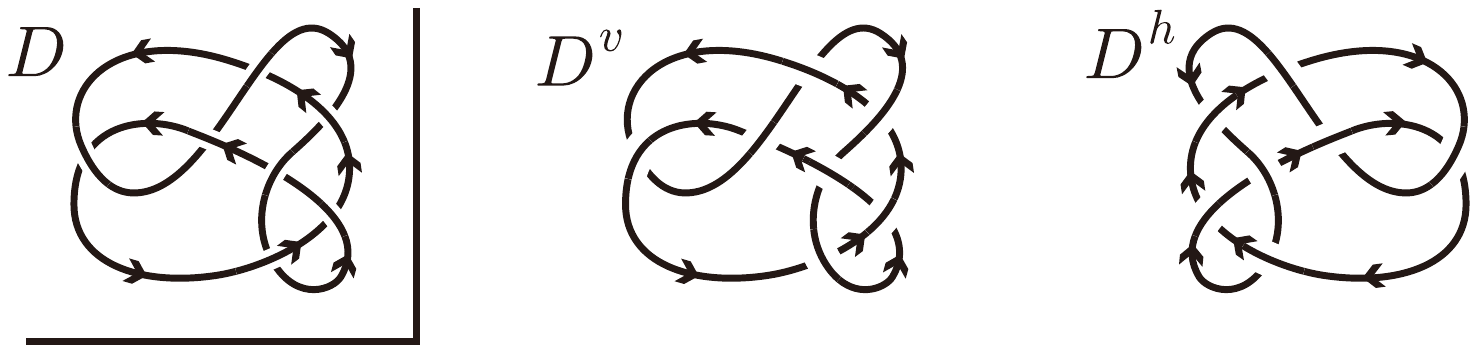} \\ \\
\includegraphics[width=0.8\textwidth, pagebox=artbox]{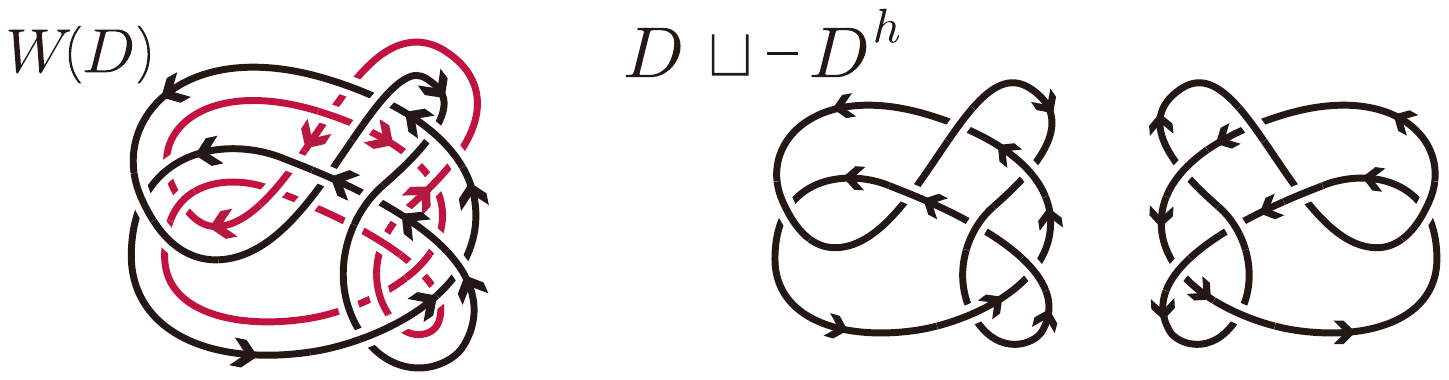}
\end{tabular}
\caption{Mirrored diagrams and the doubled diagram}\label{fig:mirror-and-double}
\end{figure}

\begin{remark}
For an oriented virtual link diagram $D$, we can define diagrams 
$-D$, $D^v$, $D^h$ and $W(D)$ in a similar way. 
However, since $D^h$ and $D^v$ may represent different virtual links,  
$W(D)$ and $D \sqcup -D^h$ may also represent different ones.
\end{remark}

Let us construct a first bijection  $\Psi_1$ from $\Col{\Qdl{X}}{D}$ 
to a certain subset of $\Col{X}{W(D)}$. 
We associate each quandle coloring in  $\Col{\Qdl{X}}{D}$ 
with a biquandle coloring in $\Col{X}{W(D)}$ as follows. 
When an arc of $D$ receives an element, say $a \in \Qdl{X}$, 
by a quandle coloring in $\Col{\Qdl{X}}{D}$, 
we assign the same element $a \in X =\Qdl{X}$ (as a set) to the pair of paralleled semi-arcs of $W(D)$ 
originated from the arc except near crossings of $D$. 
In order to satisfy the condition of biquandle colorings of $W(D)$ by $X$, 
each of other semi-arcs of $W(D)$, which are originated from crossings of $D$, 
must receive a unique element in $X$. 
Then we obtain a unique biquandle coloring in $\Col{X}{W(D)}$. 
See Figure~\ref{fig:double-color}. 
Hence we now have a (beautiful and natural) injection 
from $\Col{\Qdl{X}}{D}$ to $\Col{X}{W(D)}$. 
To describe its image, 
let $\ColD{X}{W(D)}$ denote the set of all biquandle colorings of $W(D)$ by $X$ 
such that each pair of paralleled semi-arcs of $W(D)$ 
receives the same color except near crossings, which are originated from crossings of $D$, 
as in Figure~\ref{fig:double-color}. 
We should notice that, for such a biquandle coloring, the condition which the colors of the four pairs of the semi-arcs of $W(D)$ 
originated from one crossing of $D$ have to satisfy is equivalent to 
the quandle coloring condition at the original crossing of $D$. 
Then the above injection has $\ColD{X}{W(D)}$ as its image and 
induces a bijection $\Psi_1 \colon \Col{\Qdl{X}}{D} \to \ColD{X}{W(D)}$. 
\begin{figure}[thbp]
\includegraphics[width=0.8\textwidth, pagebox=artbox]{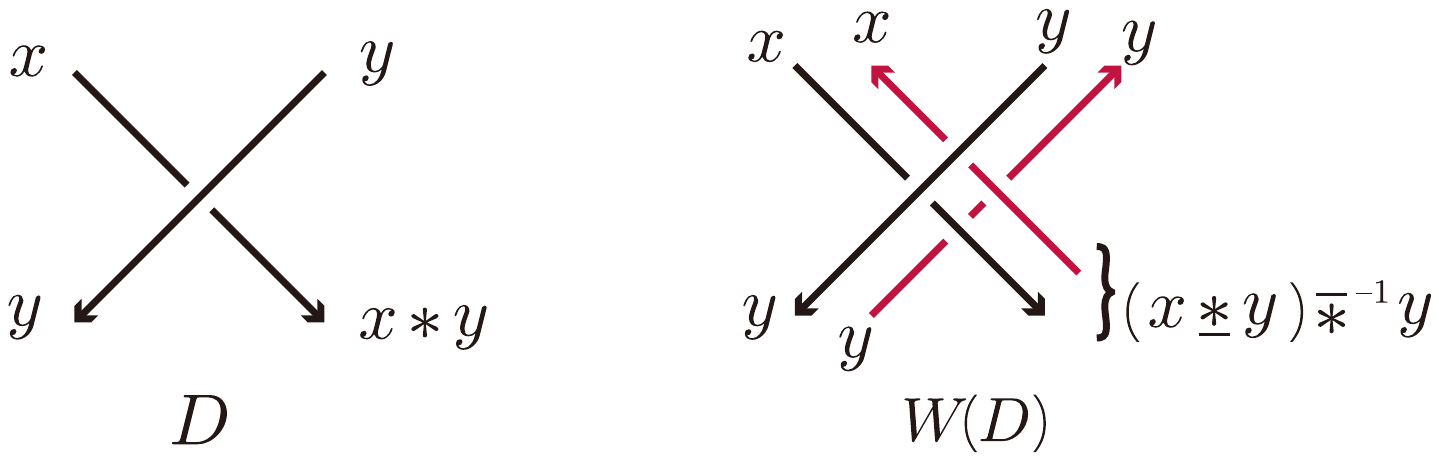}
\caption{The map $\Psi_1$ is 
\lq\lq taking the doubled diagram\rq\rq}\label{fig:double-color}
\end{figure}

Let us construct a second bijection $\Psi_2$ from $\ColD{X}{W(D)}$
to a certain subset of $\Col{X}{D \sqcup -D^h}$.  
Since $W(D)$ and  $D \sqcup -D^h$ represent the same oriented link, 
we have the bijection
$\Psi'_2 \colon \Col{X}{W(D)} \to \Col{X}{D \sqcup -D^h}$ 
associated with a sequence of Reidemeister moves 
that flips $-D^v$ to $-D^h$, and at the same time, pulls away $-D^v$ from $W(D)$:
more precisely, let $L$ and $-L^v$ be oriented links respectively represented by $D$ and $-D^v$ and assume that $L$ (resp.~$-L^v$) is contained in a ball in $\mathbb{R}^3_+$ (resp.~$\mathbb{R}^3_-$) so that $L \sqcup -L^v$ is represented by $W(D)$; 
take an isotopy that gives a half rotation and horizontal translation to the ball in $\mathbb{R}^3_-$, and perturb it to obtain a sequence of Reidemeister moves taking $W(D)$ to $D \sqcup -D^h$, which defines the bijection. 
See the second row of Figure~\ref{fig:pull-away} in Section~\ref{sec:intro}. 
To describe the image of $\ColD{X}{W(D)}$, 
we observe a relation between $\Col{X}{D}$ and $\Col{X}{-D^h}$. 
For a biquandle coloring $\mathcal{C} \in \Col{X}{D}$, 
the composite $\mathcal{C} \circ m^{-1}$, denoted by $\mathcal{C}^*$, 
turns out to be a biquandle coloring in $\Col{X}{-D^h}$ by Figure~\ref{fig:mirror-color}, 
where we regard the involution $m \colon \R^2 \to \R^2$ 
as the natural bijection from $\mathcal{SA}(D)$ to 
$\mathcal{SA}(D^h) = \mathcal{SA}(-D^h)$ by abuse of notation.  
We define $\mathcal{C} \sqcup \mathcal{C}^* \in \Col{X}{D \sqcup -D^h}$ by 
\[
(\mathcal{C} \sqcup \mathcal{C}^*) |_{\mathcal{SA}(D)} = \mathcal{C} 
\text{\quad and \quad}
(\mathcal{C} \sqcup \mathcal{C}^*) |_{\mathcal{SA}(-D^h)} = \mathcal{C}^* , 
\]
and $\ColD{X}{D \sqcup -D^h}$ to be the set of biquandle colorings  
$\mathcal{C} \sqcup \mathcal{C}^*$ for all $\mathcal{C} \in \Col{X}{D}$. 
As we see later, 
the bijection $\Psi'_2$ is independent of the choice of the Reidemeister moves (Lemma \ref{lem:well-def}) and the image of $\ColD{X}{W(D)}$ is equal to $\ColD{X}{D \sqcup -D^h}$ (Lemma \ref{lem:im-psi2}). 
Then, we denote the restriction by $\Psi_2 \colon \ColD{X}{W(D)} \to \ColD{X}{D \sqcup -D^h}$.
\begin{figure}[thbp]
\includegraphics[width=0.7\textwidth, pagebox=artbox]{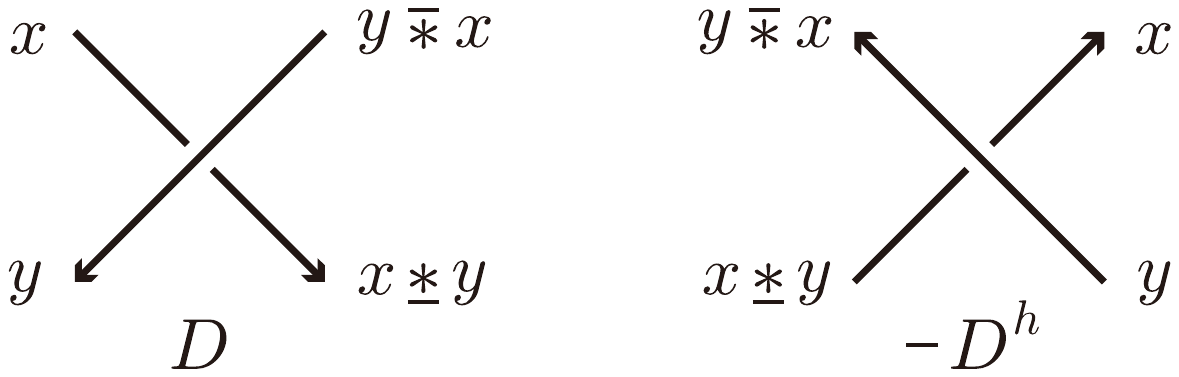}
\caption{Natural coloring for $-D^h$}\label{fig:mirror-color}
\end{figure}

Let us construct a third bijection 
$\Psi_3$ from $\ColD{X}{D \sqcup -D^h}$ to $\Col{X}{D}$. 
Let $\Psi'_3: \Col{X}{D \sqcup -D^h} \to \Col{X}{D}$ denote the restriction. 
Then, we define $\Psi_3 \colon \ColD{X}{D \sqcup -D^h} \to \Col{X}{D}$ to be the bijection induced by $\Psi'_3$.

Finally, we define $\Psi: \Col{\Qdl{X}}{D} \to \Col{X}{D}$ to be the composite $\Psi'_3 \circ \Psi'_2 \circ \Psi_1.$ Although the well-definedness of $\Psi'_2$ is not shown, we can define $\Psi'_2$ (and hence $\Psi$) by fixing a sequence of Reidemeister moves as above. After identifying the image of $\ColD{X}{W(D)}$ under $\Psi'_2$, $\Psi$ can be defined as $\Psi_3 \circ \Psi_2 \circ \Psi_1$ and then we immediately find $\Psi$ to be a bijection. However, we need some algebraic arguments for the well-definedness, 
and in Subsection~\ref{subsec:proof} 
we give a diagrammatic proof of bijectivity without such an argument.

\subsection{From biquandle colorings to quandle colorings}\label{subsec:bqtoq}
For an oriented link diagram $D$ and a biquandle $X$, 
we give a definition of a bijection $\Phi \colon \Col{X}{D} \to \Col{\Qdl{X}}{D}$ 
by pushing out each semi-arc to the unbounded region
and reading off the color of the resultant semi-arc as explained below. 
See Figure~\ref{fig:push-out}.

Let $\mathcal{C}$ be a biquandle coloring of $D$ by a biquandle $X$. 
For each semi-arc $\sigma$ of $D$, there are the two regions divided by $\sigma$ 
and a normal vector of $\sigma$ directs from one to the other; 
the former region is called the \textit{specified region} of $\sigma$. 
Take an arc $\gamma$ from the unbounded region 
to the semi-arc $\sigma$, which misses crossings, intersects semi-arcs $\tau_1,\ldots,\tau_k$ transversely in this order, and gets to $\sigma$ from the side of the specified region; for simplicity, we assume $\gamma$ to have no self-crossing. 
Then we push out the semi-arc $\sigma$ along the arc $\gamma$ and read off the color, 
say $a$ in $X( = \Qdl{X})$, of the resultant semi-arc $\tilde{\sigma}$, 
whose specified region equals the unbounded region, as in Figure~\ref{fig:push-out}. 
We note that the color is uniquely determined by the colors 
of the semi-arcs $\sigma$, $\tau_1, \ldots ,\tau_k$. 
Then, we put $( \Phi(\mathcal{C}) ) (\sigma) = a$. 
As shown below, $( \Phi(\mathcal{C}) )(\sigma) \in \Qdl{X}$ does not depend on the choice of $\gamma$ (Lemma \ref{lem:phi1}), and the map $\Phi(\mathcal{C}): \mathcal{SA}(D) \to \Qdl{X}$ defines a $\Qdl{X}$-coloring (Lemma \ref{lem:phi2}). 
Thus, we obtain a map $\Phi: \Col{X}{D} \to \Col{\Qdl{X}}{D}$. 
\begin{figure}[thbp]
\includegraphics[width=0.95\textwidth, pagebox=artbox]{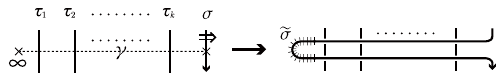}
\caption{Pushing out}\label{fig:push-out}
\end{figure}

\begin{lem}\label{lem:phi1}
The element $a = ( \Phi(\mathcal{C}) ) (\sigma)$ is independent of the choice of $\gamma$.
\end{lem}
\begin{proof}
Let $\gamma'$ be another arc from the unbounded region to $\sigma$. We may assume that $\gamma$ and $\gamma'$ are coincident in a neighborhood of the end point ($\in \sigma$). We take an isotopy which takes $\gamma$ to $\gamma'$, fixes the neighborhood, and moves the starting point in the unbounded region. As in the definition of $a \in X (= \Qdl{X})$ above, we push out $\sigma$ along $\gamma$, where the color of the head part is $a$. Then we transform the deformed portion of $\sigma$ by the isotopy and obtain the push-out of $\sigma$ along $\gamma'$. During this isotopy, the head portion of the arc always faces the unbounded region and does not go over or under another arc. Then the color of the resultant semi-arc, which equals to $\Phi(\mathcal{C})(\sigma)$ defined by using $\gamma'$, is $a$, as required.
\end{proof}

\begin{lem}\label{lem:phi2}
The defined map $\Phi(\mathcal{C}): \mathcal{SA}(D) \to \Qdl{X}$ factors through the set $\mathcal{A}(D)$ of the arcs, and the obtained map $\mathcal{A}(D) \to \Qdl{X}$ is a $\Qdl{X}$-coloring.
\end{lem}
\begin{proof}
Let $\sigma$ and $\sigma'$ be adjacent semi-arcs belonging to the same arc. To define $( \Phi(\mathcal{C}) )(\sigma)$ and $( \Phi(\mathcal{C}) )(\sigma')$, we take parallel arcs $\gamma$ and $\gamma'$ from the unbounded region to $\sigma$ and $\sigma'$. Then, the colors of the resultant semi-arcs of the push-outs along $\gamma$ and $\gamma'$ are by definition $( \Phi(\mathcal{C}) )(\sigma)$ and $( \Phi(\mathcal{C}) )(\sigma')$, respectively. Here, we can further push out the remaining part between the two pushed-out portions to connect the two heads directly, as shown in Figure \ref{fig:poa}, and hence we have $( \Phi(\mathcal{C}) )(\sigma) = ( \Phi(\mathcal{C}) )(\sigma')$.

A proof of the second part is similar: Around a crossing point, we push out the two under arcs and the over arc in a parallel way, and then we further push out the remaining part of the under arc as shown in Figure \ref{fig:poc}. We can easily check the quandle coloring condition from the biquandle coloring condition. 
\begin{figure}[thbp]
\includegraphics[width=0.75\textwidth, pagebox=artbox]{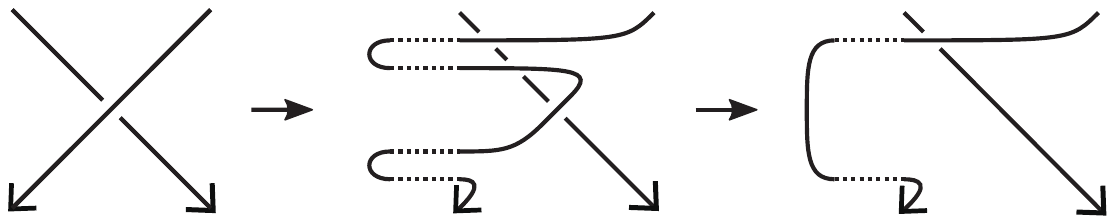}
\caption{Push-out of an over-arc}\label{fig:poa}
\end{figure}
\begin{figure}[thbp]
\begin{minipage}{0.75\textwidth}
\includegraphics[width=\textwidth, pagebox=artbox]{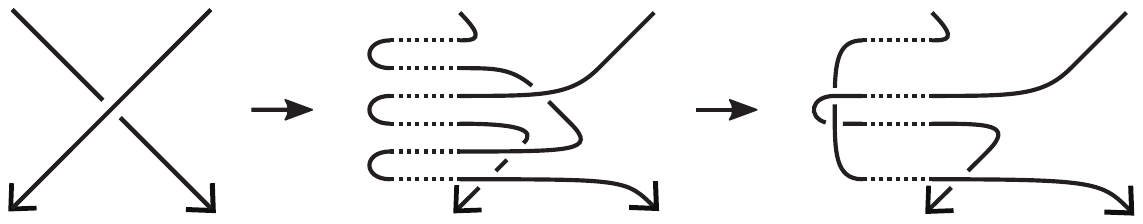}
\end{minipage}
\qquad
\begin{minipage}{0.13\textwidth}
\includegraphics[width=\textwidth, pagebox=artbox]{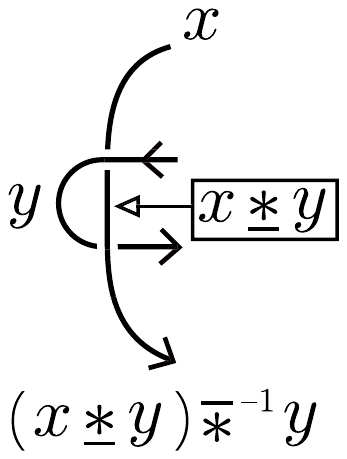}
\end{minipage}
\caption{Push-out of a crossing}\label{fig:poc}
\end{figure}
\end{proof}

\subsection{Proof of Theorem~\ref{thm:main}}\label{subsec:proof}
In order to prove our main theorem (Theorem~\ref{thm:main}), 
all we have to do is to explain that 
the \lq\lq pulling away\rq\rq\ map $\Psi$ is the inverse of 
the \lq\lq pushing out\rq\rq\ map $\Phi$.

Before giving a proof, we fix a sequence of Reidemeister moves in the definition of $\Psi'_2$. First, we slice the diagram $D$ into horizontal strips, some of which are the trivial tangle of vertical strings, so that every semi-arc has a nonempty intersection with at least one of the strips of trivial tangles 
as shown in the upper left of Figure~\ref{fig:phi2}. 
This also slices the doubled diagram $W(D)$ as shown in the lower left of Figure~\ref{fig:phi2}. 
Then we flip and pull away only the trivial-tangle parts of $-D^v$ from those of $W(D)$: 
we move the leftmost string of $-D^v$ to the right space, 
the second string to the left of the first, 
and so on, as shown in the right column of Figure~\ref{fig:phi2}. 
For example, 
for the three trivial-tangle parts of $W(D)$ shown in the lower left of Figure~\ref{fig:phi2}, 
each of the first and third parts changes into the upper right of the figure, 
and the second part changes into the lower right of the figure. 
After this, we deform the remaining parts and obtain the diagram $D \sqcup -D^h$. We should remark that, when $W(D)$ is colored by a biquandle $X$, the color of every semi-arc in the resultant diagram $D \sqcup -D^h$ is determined by the first move of the trivial-tangle parts. In other words, the resultant coloring can be calculated in each strip.
\begin{figure}[htbp]
\begin{tabular}{ccr}
\includegraphics[width=0.35\textwidth, pagebox=artbox]{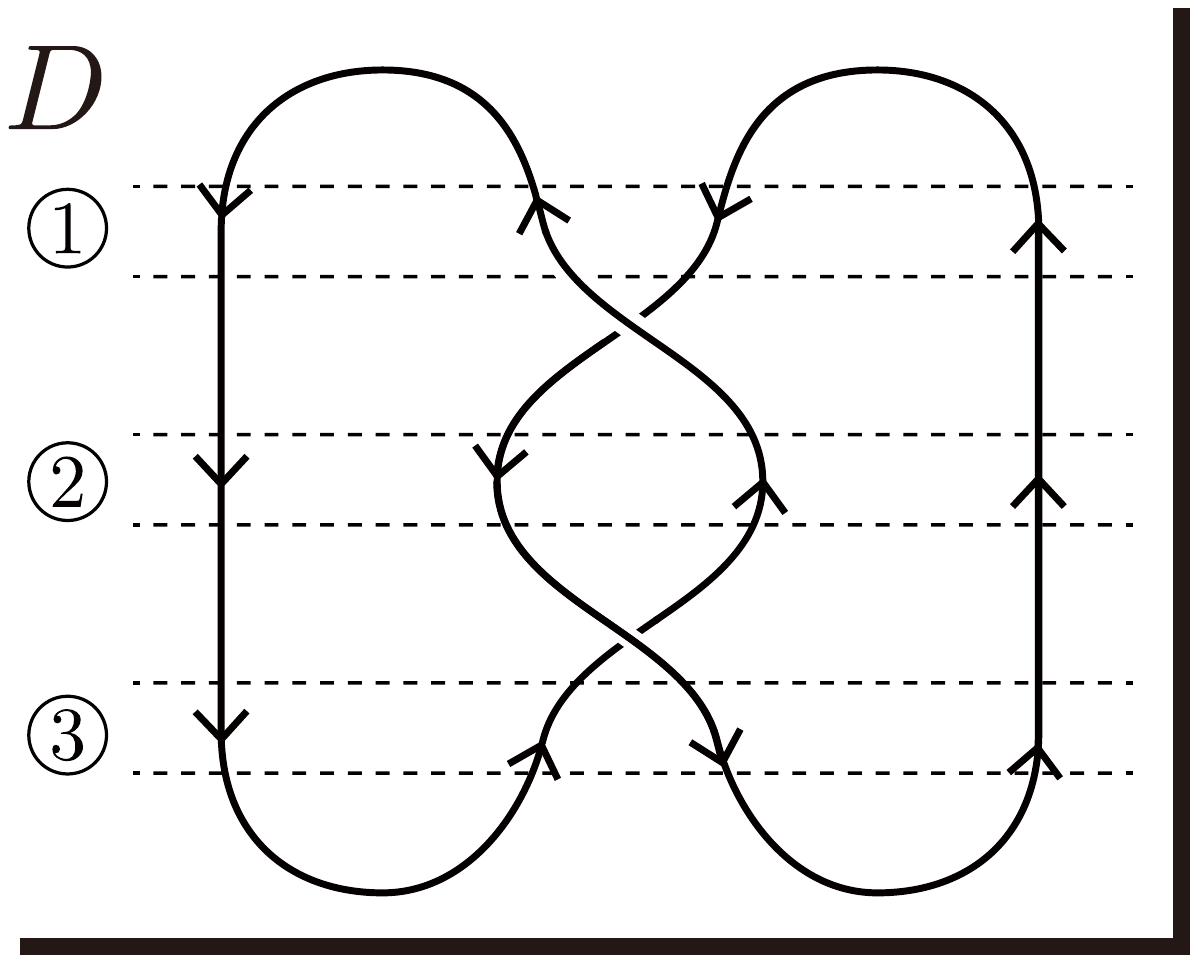}
& &
\includegraphics[width=0.50\textwidth, pagebox=artbox]{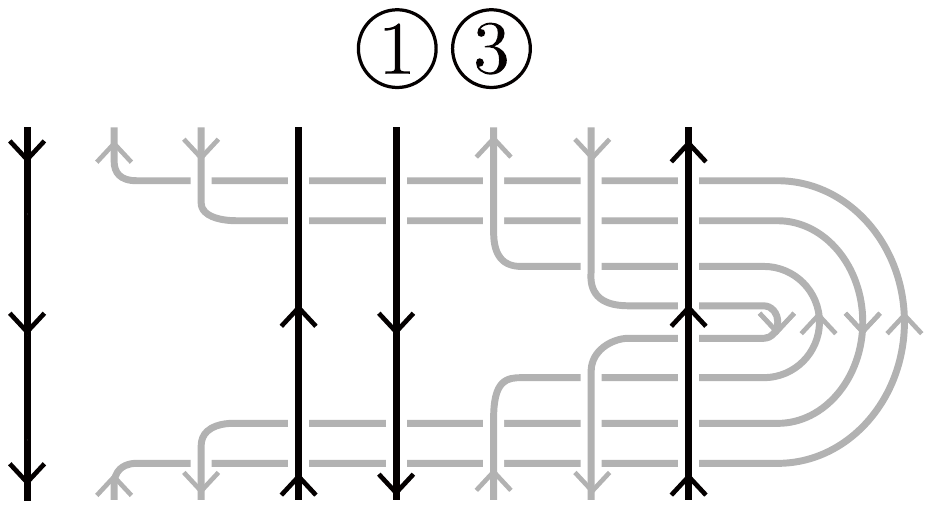}
\\ \\
\includegraphics[width=0.40\textwidth, pagebox=artbox]{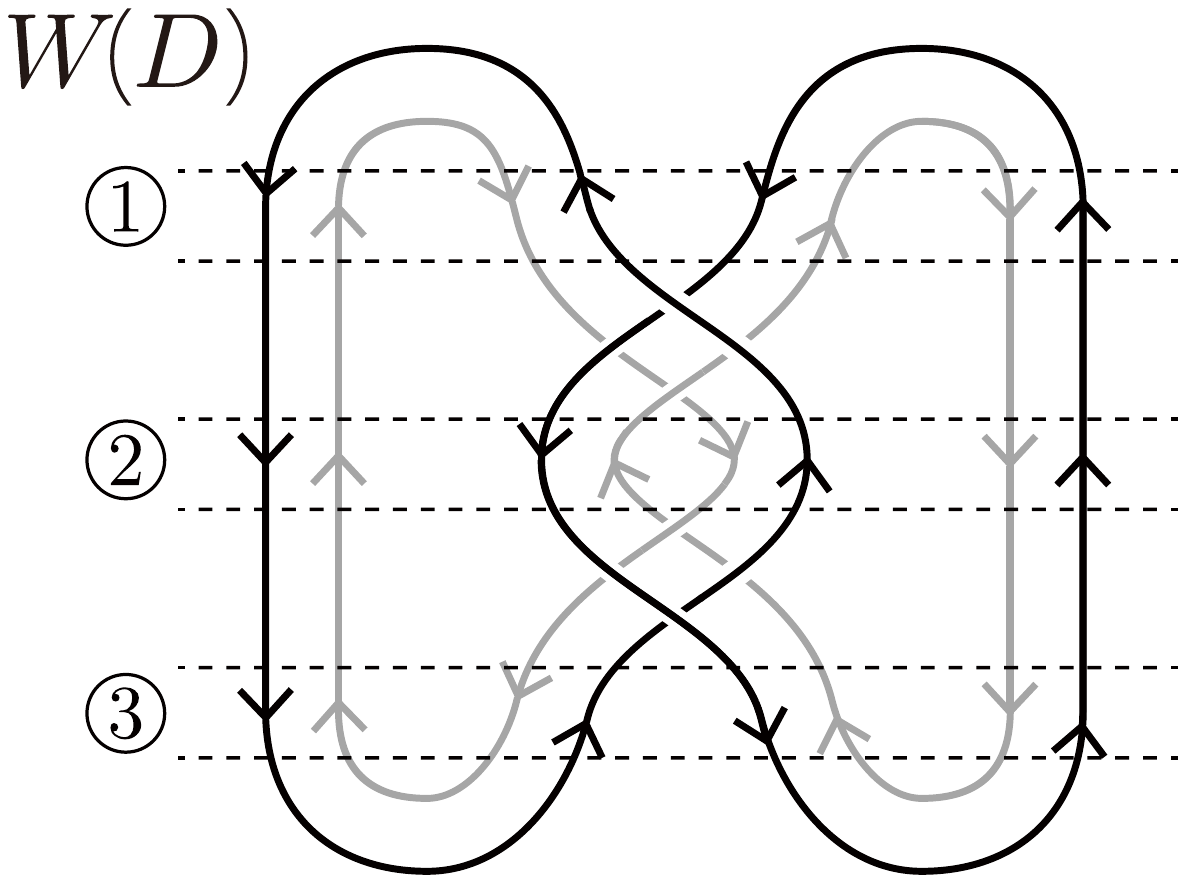}
& &
\includegraphics[width=0.50\textwidth, pagebox=artbox]{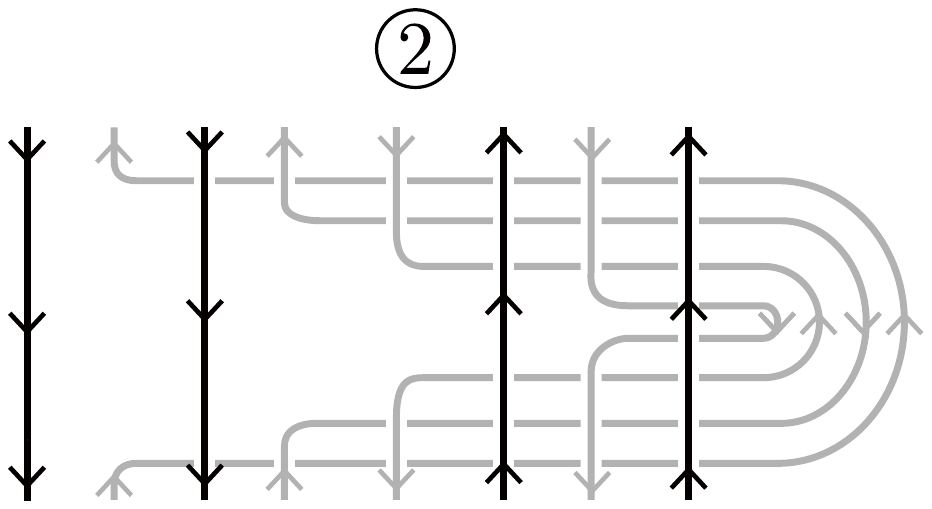}
\end{tabular}
\caption{The map $\Psi'_2$ is \lq\lq pulling away\rq\rq}\label{fig:phi2}
\end{figure}

\begin{proof}[Proof of Theorem \ref{thm:main}]
We first show that $\Phi \circ \Psi$ equals the identity map. Let $\mathcal{C} \in \Col{\Qdl{X}}{D}$ be a $\Qdl{X}$-coloring on $D$. We slice $D$ and $W(D)$ as above and take a horizontal strip of a trivial tangle. We identify each vertical string with the semi-arc $\sigma$ which it belongs to. After we pull away the strings of $-D^v$, $(\Psi(\mathcal{C}))(\sigma) \in X$ is by definition the color of the remaining semi-arc of the parallel strings corresponding to $\sigma$. Then, we push out the semi-arc to the left side. If $\sigma$ is oriented downward, the color of the resultant semi-arc is $(\Phi \circ \Psi(\mathcal{C}))(\sigma)$, as shown in Figure \ref{fig:phipsi}. Otherwise, we add a half twist to the pushed-out part and read off the color; we should note that this is equivalent to ``making a U-turn'' when we push it out. See Figure \ref{fig:U-turn}.  
In order to see that $(\Phi \circ \Psi(\mathcal{C}))(\sigma) = \mathcal{C}(\sigma),$ 
we undo the changes for \lq\lq $-D^v$-part\rq\rq\ as in Figure~\ref{fig:phipsi2}, 
that is, we put the right most \lq\lq $-D^h$-part\rq\rq\  
in the right of Figure~\ref{fig:phipsi} back to the original position. 
Since the two strings of each pair have the same color and opposite orientations, 
going over them does not change the color of the string. 
Thus, we have $(\Phi \circ \Psi(\mathcal{C}))(\sigma) = \mathcal{C}(\sigma),$ 
as asserted. 

To complete the theorem, we prove that $\Phi \colon \Col{X}{D} \to \Col{\Qdl{X}}{D}$ is injective. Let us slice $D$ and consider a horizontal strip of a trivial tangle. We push out all the vertical strings as above (see Figure \ref{fig:hpo}). When we give an $X$-coloring to the original tangle, the resultant $\Qdl{X}$-coloring appears as the colors of the leftmost semi-arcs; this is the definition of $\Phi$. Conversely, any colors given to the leftmost semi-arcs uniquely extend to a whole $X$-coloring of the tangle. This implies that the correspondence $\Phi$ is injective. 
\begin{figure}[htbp]
\centering
\includegraphics[pagebox=artbox]{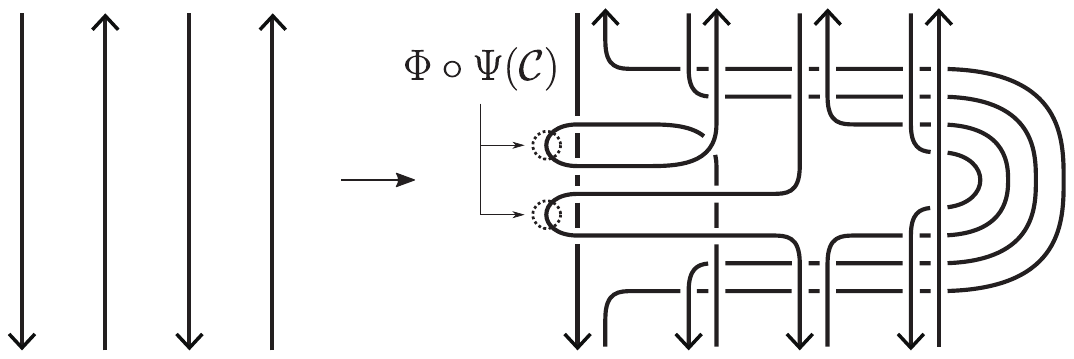}
\caption{$\Phi \circ \Psi$ on a trivial tangle (1)}\label{fig:phipsi}
\end{figure}
\begin{figure}[htbp]
\centering
\includegraphics[pagebox=artbox]{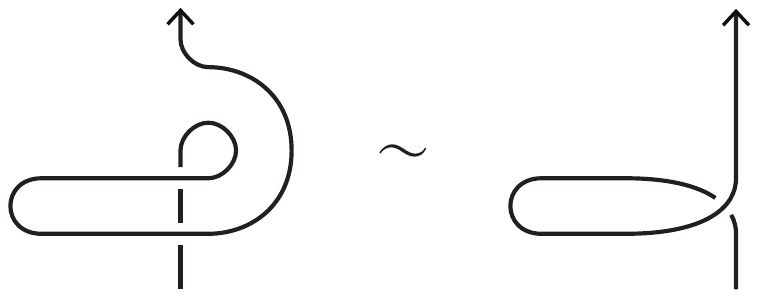}
\caption{A U-turn is equivalent to a half twist}\label{fig:U-turn}
\end{figure}
\begin{figure}[htbp]
\centering
\includegraphics[pagebox=artbox]{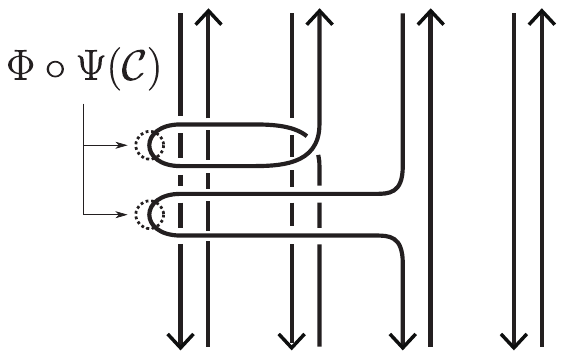}
\caption{$\Phi \circ \Psi$ on a trivial tangle (2)}\label{fig:phipsi2}
\end{figure}
\begin{figure}[htbp]
\centering
\includegraphics[pagebox=artbox]{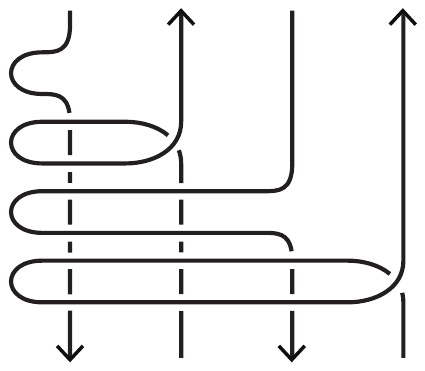}
\caption{The leftmost colors determine the whole coloring}\label{fig:hpo}
\end{figure}
\end{proof}

\par In the above proof, we saw that going over pairs of strings does not change its color in order to show $\Phi \circ \Psi = {\rm id}.$ Similar consideration for a string going under pairs of others identifies the image of $\Psi_2$:

\begin{lem}\label{lem:im-psi2}
The image of $\ColD{X}{W(D)}$ under $\Psi'_2$ is equal to $\ColD{X}{D \sqcup -D^h}$.
\end{lem}

\begin{proof}
Compare the pulled-away semi-arcs with the remaining ones, as above.
\end{proof}

\section{Algebraic description of the one-to-one correspondence}\label{sec:alg}
The shadow coloring of a biquandle coloring is recalled in Subsection \ref{subsec:shadow}. Then, we give algebraic descriptions of the correspondences $\Psi$ and $\Phi$ in Subsections \ref{subsec:alg-psi} and \ref{subsec:alg-phi}, respectively. In particular, we prove Proposition \ref{prop:bqR} in Subsection \ref{subsec:alg-psi}.

\subsection{Shadow colorings}\label{subsec:shadow}
The \textit{associated group} $\mathrm{As}(X)$ of a biquandle $X$ is a group defined by 
\[
\langle
x \ (x \in X) \mid x \cdot (y \os x) = y \cdot (x \us y) \ (x,y \in X)
\rangle .
\]
There is a natural right action of $\mathrm{As}(X)$ on $X$ defined by 
\[
x \cdot y := x \us y \quad \text{and} \quad x \cdot y^{-1}:= x \us^{-1} y
\]
for $x,y \in X$, where $y$ is regarded as the element in $\mathrm{As}(X)$.

Let $\mathcal{R}(D)$ denote the set of the regions of an oriented link diagram $D$, 
and $R_\infty \in \mathcal{R}(D)$ the unbounded region of $D$. 
For a biquandle coloring $\mathcal{C} \colon \mathcal{SA}(D) \to X$ in $\Col{X}{D}$, 
the \textit{region coloring} of $D$ with respect to $\col{C}$ is a map 
$\col{C}^R \colon \mathcal{R}(D) \to \mathrm{As}(X)$ such that 
\[ \col{C}^R(R_\infty) = e 
\quad \text{and} \quad 
\col{C}^R(R_1) \cdot \col{C}(\sigma) = \col{C}^R(R_2) \] 
for each semi-arc $\sigma \in \mathcal{SA}(D)$, 
where $R_1$ and $R_2$ are the regions divided by $\sigma$ and $R_1$ is the specified region of $\sigma$, that is, a normal vector of $\sigma$ directs from $R_1$ to $R_2$. 
We note that the region coloring $\col{C}^R$ is uniquely determined from the biquandle coloring $\col{C}$. 
Then, we define a map 
$\col{C}^S \colon \mathcal{SA}(D) \to \mathrm{As}(X) \times X$, 
called the \textit{shadow coloring} of $D$ with respect to $\col{C}$, by 
\[
\col{C}^S(\sigma) := \bigl( \col{C}^R(\rho_\sigma), \col{C}(\sigma) \bigr)
\]
for each semi-arc $\sigma \in \mathcal{SA}(D)$, 
where $\rho_\sigma$ is the specified region of $\sigma$.
We also note that a shadow coloring $\col{C}^S$ is uniquely determined from 
the biquandle coloring $\col{C}$. 

The notions of the associated group, region colorings and shadow colorings 
are defined in the same way for a quandle by regarding it as a biquandle, 
as in Remark \ref{rem:qbq}.

\subsection{Algebraic description of $\Psi$}\label{subsec:alg-psi}
We define a map $\Psi'$ 
from quandle colorings $\Col{\Qdl{X}}{D}$ 
to biquandle colorings $\Col{X}{D}$. 
Let $k \colon X \to X$ be the unique bijection 
satisfying 
\[
k(x) \os k(x) \,  \bigl( =k(x) \us k(x) \bigr) \, = x
\]
for any $x \in X$ (see, e.g., \cite{iikkmo,is}). 
We note that the bijection $k$ essentially appeared in \cite{fjk}.  
Let $\psi \colon \mathrm{As}(\Qdl{X}) \times \Qdl{X} \to X$ 
be a map defined (inductively) by 
\[\begin{array}{rcl}
\psi(e, a) & := &  a, \\
\psi(p \cdot b, a) & := &  \psi(p,a) \os \psi(p,b), \\
\psi(p \cdot b^{-1}, a) & := & \psi(p,a) \os^{\,-1} k(\psi(p,b))
\end{array}\]
for $p \in \mathrm{As}(\Qdl{X})$ 
and $a,b \in \Qdl{X}$, where $b$ is regarded as the element in $\mathrm{As}(\Qdl{X})$. 
Since the proof of well-definedness of the map $\psi$ is similar to that of 
\lq\lq the map $\psi: \mathcal{B}(Q) \to \mathcal{B'}(Q)$ 
defined in \cite[Subsection 3.2]{is}\rq\rq, we omit the details. 
See Lemmas 3.3 and 3.6 in \cite{is}.

For a quandle coloring $\mathcal{C} \colon \mathcal{A}(D) \to \Qdl{X}$ in $\Col{\Qdl{X}}{D}$, 
let
\[
\col{C}^S \colon \mathcal{SA}(D) \to \mathrm{As}(\Qdl{X}) \times \Qdl{X}
\]
be the shadow coloring with respect to $\mathcal{C}$ as in Subsection~\ref{subsec:shadow}.
Then a map $\Psi'(\mathcal{C}) \colon \mathcal{SA}(D) \to X$ 
is defined to be $\Psi'(\mathcal{C}) = \psi \circ \col{C}^S$.

\begin{lem}\label{lem:psi}
The map $\Psi'(\mathcal{C})$ is a biquandle coloring in $\Col{X}{D}$. 
Hence we have a map $\Psi'$ from $\Col{\Qdl{X}}{D}$ to $\Col{X}{D}$. 
\end{lem}

\begin{proof}
We can check that the map $\psi$ satisfies 
\begin{equation}\label{eq:us}
\psi(p,a) \us \psi(p,b) = \psi(p \cdot b, a \s b)
\end{equation}
for $p \in \mathrm{As}(\Qdl{X})$ and $a, b \in \Qdl{X}$, 
where $b$ is regarded as the element in $\mathrm{As}(\Qdl{X})$. 
Since the proof is similar to that of \cite[Lemma 3.4]{is}, we omit the details. 
We find $\Psi'(\mathcal{C})$ to be a biquandle coloring in $\Col{X}{D}$ from 
the equation (\ref{eq:us}) and the definition of $\psi$.
\end{proof}

\begin{lem}\label{lem:psi=psi}
The correspondence $\Psi'$ is equal to $\Psi$.
\end{lem}
\begin{proof}
As we saw in Subsection \ref{subsec:proof}, it is sufficient to check the equality in the case of a trivial tangle $T$ of vertical strings. That is, let $\sigma_i$ be the $i$-th string from the left, and we give a color $a_i = \mathcal{C}(\sigma_i) \in \Qdl{X}$ to $\sigma_i$. We put $\epsilon_i = 1$ (resp.~$-1$) if it is oriented downward (resp.~upward). We consider the double $W(T)$ of $T$ and give the paralleled coloring. After the pull-away (and the flip) of $-T^v$ as in Subsection \ref{subsec:proof}, the color of the $i$-th remaining semi-arc is by definition $(\Psi(\mathcal{C}))(\sigma_i)$, denoted by $x_i$. On the other hand, we define $\col{C}^S$ as above, regarding the leftmost space as the unbounded region, and $\Psi'(\mathcal{C})$ to be $\psi \circ \col{C}^S$. We denote $(\Psi'(\mathcal{C}))(\sigma_i)$ by $x'_i$. We shall show $x_i = x'_i$ inductively, as follows.

If $\epsilon_1 = 1$, $\sigma_1$ does not intersect with any other arc during the pull-away. Then we have $x_1 = a_1$. In this case, the color of the specified region of $\sigma_1$ is $e \in \mathrm{As}(\Qdl{X})$, and hence $x'_1 = \psi(e,a_1) = a_1$, which shows $x_1=x'_1.$ If $\epsilon_1 = -1$, the copy of $\sigma_1$ with the reversed orientation goes under $\sigma_1$, and hence we find $x_1 = k(a_1)$. In this case, the color of the specified region of $\sigma_1$ is $a_1^{-1} \in \mathrm{As}(\Qdl{X})$, and then $x'_1 = \psi(a_1^{-1},a_1) = a_1 \os^{\, -1} k(a_1) = k(a_1).$ Thus, we have $x_1 = x'_1$.

Assume that $x_i = x'_i$ holds for $i = 1, \dots, j-1$, and then let us show $x_j = x'_j.$ 
We consider the case where $\epsilon_j=1$; the other case, where $\epsilon_j=-1$, can be shown in a similar way. 
Here, the left $j-1$ arcs of $-T^v$ go under $\sigma_j$ and change the color of $\sigma_j$. The $i$-th pulled-away arc ($i<j$) possibly has an intersection with $\sigma_i$ (if $\epsilon_i = -1$) first, but thereafter it goes under the pairs of arcs, which does not change its color, as we saw in the proof of Lemma \ref{lem:im-psi2}. Then, we have $x_j = (((a_j \os^{\, \epsilon_1} x_1) \os^{\, \epsilon_2} x_2) \cdots ) \os^{\, \epsilon_{j-1}} x_{j-1}.$ On the other hand, we easily find the color, denoted by $p_i \in \mathrm{As}(\Qdl{X})$, of the specified region of $\sigma_i$ in $T$ to be
\[ p_i = \left\{ \begin{array}{ll} a_1^{\epsilon_1}\cdots a_{i-1}^{\epsilon_{i-1}} & \text{if $\epsilon_i = 1$},\\
a_1^{\epsilon_1}\cdots a_i^{\epsilon_i} & \text{if $\epsilon_i = -1$}, \end{array} \right.\]
and then have $x'_i = \psi(p_i,a_i)$. By these formulas and the definition of $\psi$, we have
\begin{align*}
x_j 
&= (((a_j \os^{\,\epsilon_1} x_1) \os^{\,\epsilon_2} x_2) \cdots ) \os^{\,\epsilon_{j-1}} x_{j-1}\\
&= (((a_j \os^{\,\epsilon_1} x'_1) \os^{\,\epsilon_2} x'_2) \cdots ) \os^{\,\epsilon_{j-1}} x'_{j-1}\\
&= (((\psi(e,a_j) \os^{\,\epsilon_1} \psi(p_1,a_1)) \os^{\,\epsilon_2} \psi(p_2,a_2)) \cdots ) \os^{\,\epsilon_{j-1}} \psi(p_{j-1},a_{j-1})\\
&= ((\psi(a_1^{\epsilon_1},a_j) \os^{\,\epsilon_2} \psi(p_2,a_2)) \cdots ) \os^{\,\epsilon_{j-1}} \psi(p_{j-1},a_{j-1})\\
&= ((\psi(a_1^{\epsilon_1}a_2^{\epsilon_2},a_j) \os^{\,\epsilon_3} \psi(p_3,a_3)) \cdots ) \os^{\,\epsilon_{j-1}} \psi(p_{j-1},a_{j-1})\\
&= \cdots\\
&= \psi(a_1^{\epsilon_1} \cdots a_{j-1}^{\epsilon_{j-1}},a_j)\\
&= x'_j
\end{align*}
as required, since we assume $x_i = x'_i$ for $i < j$. 
\end{proof}

\begin{lem}\label{lem:compatibility}
The map $\Psi$ is compatible with Reidemeister moves. That is, given a sequence of Reidemeister moves which takes $D$ to $D'$, the following diagram commutes\rm{:}
\[ \begin{CD}
\Col{\Qdl{X}}{D} @>>> \Col{\Qdl{X}}{D'} \\
@V\Psi VV @VV\Psi V \\
\Col{X}{D} @>>> \Col{X}{D}
\end{CD} \]
\end{lem}

\begin{proof}
According to Lemma \ref{lem:psi=psi}, the lemma is equivalent to the corresponding claim for $\Psi'$, but this is trivial. 
\end{proof}

\begin{proof}[Proof of Proposition \ref{prop:bqR}]
Let $f_1^Q,f_2^Q: \Col{\Qdl{X}}{D} \to \Col{\Qdl{X}}{D'}$ be the two induced bijections on the quandle colorings. As we recalled in 
Subsection~\ref{subsec:same_R-move}, 
the proposition holds for quandles and hence $f_1^Q = f_2^Q$. By Lemma \ref{lem:compatibility}, we find that
\[ f_1^B = \Psi \circ f_1^Q \circ \Psi^{-1} = \Psi \circ f_2^Q \circ \Psi^{-1} = f_2^B \]
as required.
\end{proof}

Thus, $\Psi_2$ in the definition of $\Psi$ is well defined:

\begin{lem}\label{lem:well-def}
The bijection $\Psi'_2: \Col{X}{W(D)} \to \Col{X}{D \sqcup -D^h}$ is independent of the choice of Reidemeister moves representing the flip and the pull-away.
\end{lem}

\begin{proof}
This immediately follows from Proposition \ref{prop:bqR}.
\end{proof}

\subsection{Algebraic description of $\Phi$}\label{subsec:alg-phi}
We define a map $\Phi'$ from biquandle colorings $\Col{X}{D}$
to quandle colorings $\Col{\Qdl{X}}{D}$. 
Let $\phi \colon \mathrm{As}(X) \times X \to \Qdl{X}(=X)$ be a map defined by 
$\phi(g,x)=x \diamond g^{-1}$, where a right action $\diamond$ of $\mathrm{As}(X)$ on $X$ 
is defined by $x \diamond y = x \os y$ and  $x \diamond y^{-1} = x \os^{\,-1} y$. 
We should note that the right action $\diamond$ 
is different from  the usual right action denoted by ``$\cdot$'' in Subsection~\ref{subsec:shadow}. 
Although the right action $\diamond$ is used in the definition of the map $\phi$, 
the symbol $\diamond$ itself does not appear explicitly in what follows. 
%
For a biquandle coloring 
$\mathcal{C} \colon \mathcal{SA}(D) \to X$ in $\Col{X}{D}$, define
$\col{C}^S \colon \mathcal{SA}(D) \to \mathrm{As}(X) \times X$ 
as in Subsection \ref{subsec:shadow}. 
Then we can check that 
a map $\phi \circ \col{C}^S \colon  \mathcal{SA}(D) \to \Qdl{X}$ 
factors through the projection $\pi: \mathcal{SA}(D) \to \mathcal{A}(D)$. 
Hence we have a unique map, 
denoted by $\Phi'(\mathcal{C}) \colon \mathcal{A}(D) \to \Qdl{X}(=X)$, 
satisfying $\Phi'(\mathcal{C}) \circ \pi = \phi \circ \col{C}^S$.

\begin{lem}
The map $\Phi'(\mathcal{C})$ is a quandle coloring in $\Col{\Qdl{X}}{D}$. 
Hence we have a map $\Phi'$ from $\Col{X}{D}$ to $\Col{\Qdl{X}}{D}$.  
\end{lem}

\begin{proof}
For a crossing point of $D$, let $\sigma_i, \sigma_j, \sigma_k, \sigma_l \in \col{SA}(D)$ be the four semi-arcs as in Figure \ref{fig:crossing} (right) and let $g \in \mathrm{As}(X)$ be the region color of the specified region. Then we have
\begin{align*}
(\Phi'(\col{C}))(\sigma_k) 
&= \phi \left( g \cdot \col{C}(\sigma_j), \col{C}(\sigma_k) \right)\\
&= \phi \left( g \cdot \col{C}(\sigma_j), \col{C}(\sigma_i) \us \col{C}(\sigma_j) \right)\\
&= \phi \left( 
g, (\col{C}(\sigma_i) \us \col{C}(\sigma_j)) \os^{\,-1} \col{C}(\sigma_j) 
\right)\\
&= \phi \left( g, \col{C}(\sigma_i) * \col{C}(\sigma_j) \right).
\end{align*}
Since the right action of $g^{-1} \in \mathrm{As}(X)$ on $X = \Qdl{X}$ is a quandle automorphism on $\Qdl{X}$, we have
\[ (\Phi'(\col{C}))(\sigma_k) = (\Phi'(\col{C}))(\sigma_i) * (\Phi'(\col{C}))(\sigma_j) \]
as required.
\end{proof}

\begin{lem}
The correspondence $\Phi'$ is equal to $\Phi$.
\end{lem}

\begin{proof}
We take $\sigma_i$ and $\epsilon_i$ as in the proof of Lemma \ref{lem:psi=psi}; it is sufficient to show that $(\Phi(\mathcal{C}))(\sigma_j) = (\Phi'(\mathcal{C}))(\sigma_j)$ for each $j$ such that $\epsilon_j = 1$. We set $x_i = \mathcal{C}(\sigma_i)$. As we push out $\sigma_j$ to the left side, we find that $(\Phi(\mathcal{C}))(\sigma_j) = ((x_j \os^{\, -\epsilon_{j-1}} x_{j-1}) \cdots ) \os^{\, -\epsilon_1} x_1$. Also, by the definition of $\Phi',$ we have $(\Phi'(\mathcal{C}))(\sigma_j) = \phi(x_1^{\epsilon_1} \cdots x_{j-1}^{\epsilon_{j-1}}, x_j)$. Then, the definition of $\phi$ immediately shows that $(\Phi(\mathcal{C}))(\sigma_j) = (\Phi'(\mathcal{C}))(\sigma_j)$.
\end{proof}

\section{Cocycle invariants}\label{sec:ci}
We review biquandle cocycle invariants \cite{ces, kkkl} in Subsection~\ref{subsec:bqci} 
and shadow quandle cocycle invariants \cite{cegs, cks}  
in Subsection~\ref{subsec:sqci}. 
Then, in Subsection~\ref{subsec:qci}, 
we interpret biquandle cocycle invariants in terms of shadow quandle cocycle invariants 
by using the map $\psi: \mathrm{As}(\Qdl{X}) \times \Qdl{X} \to X$ defined in 
Subsection~\ref{subsec:alg-psi} for a biquandle $X$.

\subsection{Biquandle cocycle invariants}\label{subsec:bqci}
Let $X$ be a biquandle and $A$ an abelian group. 
A map $\theta \colon X^n \to A$ is called a 
\textit{biquandle $n$-cocycle} if $\theta$ satisfies the conditions 
\begin{align*}
\sum_{i=1}^{n+1} (-1)^{i-1}  \bigl( & \theta(x_1, \dots, x_{i-1}, x_{i+1}, \dots, x_{n+1})\\
& - \theta(x_1 \us x_i, \dots, x_{i-1} \us x_i, x_{i+1} \os x_i, \dots, x_{n+1} \os x_i) \bigr) = 0_A 
\end{align*}
and 
\[\theta(x_1, \dots, x_{j-1}, y,  y, x_{j+2}, \dots, x_n) = 0_A \quad (j=1, \dots, n-1)\]
for any $x_1, \dots, x_{n+1}, y \in X$.

Let $\theta: X \times X \to A$ be a biquandle 2-cocycle. For an $X$-coloring $\mathcal{C}$ on an oriented link diagram $D$, 
we associate a weight $\pm \theta(x,y)$ 
to each crossing of $D$ as in Figure~\ref{fig:bqci}, 
where $x,y \in X$ are the colors of the semi-arcs indicated in the figure. 
Then we sum them up for all the crossings of $D$ 
to obtain $(\Phi_{\theta}(D))(\mathcal{C}) \in A$. 
This defines a map $\Phi_{\theta}(D): \Col{X}{D} \to A$, 
which is called a \textit{biquandle cocycle invariant}. 
When $X$ is finite, $\Phi_{\theta}(D)$ is often considered as an element 
\[\sum_{\mathcal{C} \in \Col{X}{D}} (\Phi_{\theta}(D))(\mathcal{C}) \]
in the group ring $\Z[A]$ and does not depend on the choice of the diagram $D$. 
See \cite{ces} for the biquandle cocycle invariant of links (and 
\cite{kkkl} for that of surface links). 
\begin{figure}[thbp]
\includegraphics[width=0.50\textwidth, pagebox=artbox]{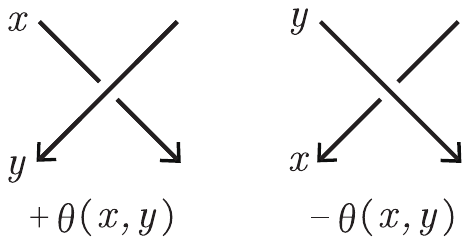}
\caption{Weights for biquandle cocycle invariants}\label{fig:bqci}
\end{figure}

\subsection{Shadow quandle cocycle invariants}\label{subsec:sqci}
%
Let $Q$ be a quandle and $A$ an abelian group. 
A map $\theta \colon \mathrm{As}(Q) \times Q^n \to A$ is called a 
\textit{shadow quandle $n$-cocycle} if $\theta$ satisfies the conditions 
\begin{align*}
\sum_{i=1}^{n+1} (-1)^{i-1}  \bigl( & \theta(p, a_1, \dots, a_{i-1}, a_{i+1}, \dots, a_{n+1})\\
& - \theta(p \cdot a_i, a_1 \s a_i, \dots, a_{i-1} \s a_i, a_{i+1}, \dots, a_{n+1}) \bigr) = 0_A 
\end{align*}
and
\[\theta(p, a_1, \dots, a_{j-1}, b, b, a_{j+2}, \dots, a_n) = 0_A \quad (j=1, \dots, n-1)\]
for any $p \in \mathrm{As}(Q)$ and $a_1, \dots, a_{n+1},b \in Q$.

Let $\theta: \mathrm{As}(Q) \times Q \times Q \to A$ be a shadow quandle $2$-cocycle. 
For a $Q$-coloring $\mathcal{C}$ on an oriented link diagram $D$ 
(and the unique shadow coloring 
$\col{C}^S \colon \mathcal{SA}(D) \to \mathrm{As}(Q) \times Q$ with respect to $\mathcal{C}$), 
we associate a weight $\pm \theta(p,a,b)$ 
to each crossing of $D$ as in Figure~\ref{fig:sqci}, 
where $p \in {\rm As}(Q)$ is the color of the region indicated in the figure. 
Then we sum them up for all the crossings of $D$ 
to obtain $(\Phi_{\theta}^{S}(D))(\col{C}) \in A$. 
This defines a map $\Phi_{\theta}^{S}(D): \Col{Q}{D} \to A$, which is called a \textit{shadow quandle cocycle invariant}. 
When $Q$ is finite, $\Phi_{\theta}^{S}(D)$ is often considered as an element 
\[\sum_{\mathcal{C} \in \Col{Q}{D}} (\Phi_{\theta}^{S}(D))(\col{C}) \]
in the group ring $\Z[A]$ and does not depend on the choice of the diagram $D$. 
We refer to \cite{kam} for more details. 
\begin{figure}[thbp]
\includegraphics[width=0.50\textwidth, pagebox=artbox]{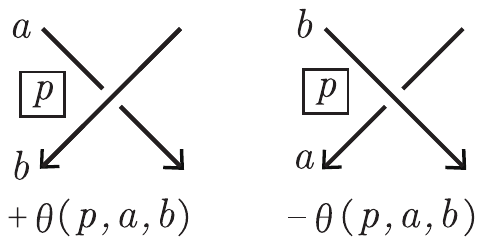}
\caption{Weights for shadow quandle cocycle invariants}\label{fig:sqci}
\end{figure}

\subsection{Relationship between two cocycle invariants}\label{subsec:qci}
For a biquandle $X$ and an abelian group $A$, 
let $\theta \colon X^n \to A$ be a biquandle $n$-cocycle. 
The pullback, $\psi^* \theta \colon \mathrm{As}(\Qdl{X}) \times \Qdl{X}^n \to A$, 
of $\theta$ by the map $\psi \colon \mathrm{As}(\Qdl{X}) \times \Qdl{X} \to X$ 
is defined by 
\[
(\psi^* \theta) (p,a_1, \dots, a_n) := \theta(\psi(p,a_1),\dots, \psi(p,a_n))
\]
for $p \in \mathrm{As}(\Qdl{X})$ and $a_1, \dots, a_n \in \Qdl{X}$. 
Then we find that $\psi^* \theta$ is a shadow quandle $n$-cocycle: for any $p \in \mathrm{As}(\Qdl{X})$ and $a_1, \dots, a_{n+1} \in \Qdl{X}$, we have 
\begin{align*}
\sum_{i=1}^{n+1} (-1)^{n-1}((\psi^* \theta&)(p,a_1,\dots,a_{i-1},a_{i+1},\dots,a_{n+1})\\
- (&\psi^*\theta)(p \cdot a_i, a_1 \s a_i, \dots, a_{i-1} \s a_i, a_{i+1},\dots,a_{n+1}))\\
=\sum_{i=1}^{n+1} (-1)^{n-1}(\theta(\psi(&p,a_1),\dots,\psi(p,a_{i-1}),\psi(p,a_{i+1}),\dots,\psi(p,a_{n+1}))\\
- \theta(&\psi(p \cdot a_i, a_1 \s a_i), \dots, \psi(p \cdot a_i, a_{i-1} \s a_i),\\
&\qquad\qquad\psi(p \cdot a_i, a_{i+1}),\dots, \psi(p \cdot a_i, a_{n+1})))\\
=\sum_{i=1}^{n+1} (-1)^{n-1}(\theta(\psi(&p,a_1),\dots,\psi(p,a_{i-1}),\psi(p,a_{i+1}),\dots,\psi(p,a_{n+1}))\\
- \theta(&\psi(p,a_1) \us \psi(p,a_i), \dots, \psi(p, a_{i-1}) \us \psi(p,a_i),\\
&\qquad\psi(p, a_{i+1}) \os \psi(p,a_i),\dots, \psi(p, a_{n+1}) \os \psi(p,a_i)))\\
= 0_A, \hspace{2.15cm}
\end{align*}
where the second equality follows from the equation (\ref{eq:us}) in the proof of Lemma~\ref{lem:psi} and the definition of $\psi$. 
In fact, using the map $\psi$, 
we can define a chain map from the shadow quandle chain complex to the biquandle chain complex and $\psi^*$ above is the dual of it.

\begin{theorem}\label{thm:bci}
For a biquandle $X$ and an abelian group $A$, 
let $\theta \colon X \times X \to A$ be a biquandle $2$-cocycle. 
Then, for a diagram $D$ of an oriented link, 
the biquandle cocycle invariant with respect to $\theta$ 
can be considered as the shadow quandle cocycle invariant with respect to 
$\psi^* \theta$. In particular, we have
\[\Phi_{\theta}(D) = \Phi_{\psi^* \theta}^{S} (D)\] 
in $\Z[A]$ when $X$ is finite. 
\end{theorem}

\begin{proof}
For a $\Qdl{X}$-coloring $\mathcal{C}: \mathcal{A}(D) \to \Qdl{X}$, 
recall that the shadow coloring $\col{C}^S \colon \mathcal{SA}(D) \to \mathrm{As}(\Qdl{X}) \times \Qdl{X}$ 
is the map defined by $\col{C}^S (\sigma) = (\mathcal{C}^R(\rho_\sigma), \mathcal{C}(\sigma))$ for $\sigma \in \mathcal{SA}(D)$. 
By the definition of $\Psi'$ and Lemma \ref{lem:psi=psi}, we have 
\[ 
\bigl(\Psi (\mathcal{C}) \bigr) (\sigma) = \psi \bigl( \col{C}^S (\sigma) \bigr)
= \psi \bigl( \mathcal{C}^R(\rho_{\sigma}), \mathcal{C}(\sigma) \bigr)
\]
for $\sigma \in \mathcal{SA}(D)$. 
This implies that at every crossing of $D$, the weight with respect to $\col{C}^S$ 
and $\psi^* \theta$ equals that with respect to $\Psi(\mathcal{C})$ and $\theta$. 
Hence we have
\begin{align*}
\Phi_{\theta}(D) &= \sum_{\mathcal{C} \in \Col{X}{D}} \bigl( \Phi_{\theta}(D) \bigr) (\mathcal{C})\\
&= \sum_{\mathcal{C} \in \text{Col}_{\Qdl{X}}^{}(D)} \bigl( \Phi_{\theta}(D) \bigr) (\Psi (\mathcal{C}))\\
&= \sum_{\mathcal{C} \in \text{Col}_{\Qdl{X}}^{}(D)}  \bigl( \Phi_{\psi^* \theta}^{S}(D) \bigr) (\col{C}
)\\
&= \Phi_{\psi^* \theta}^{S} (D), 
\end{align*}
as required. 
\end{proof}

\section{Homotopy invariants}\label{sec:hi}
We define biquandle spaces and biquandle homotopy invariants 
in Subsection~\ref{subsec:BX}, 
and extend the two maps $\psi$ and $\phi$ 
in Subsection~\ref{subsec:one-to-one}. 
Then, we show that biquandle homotopy invariants and quandle 
homotopy invariants are equivalent in Subsection~\ref{subsec:equiv}

\subsection{Biquandle spaces and homotopy invariants}\label{subsec:BX}
For a biquandle $X$, which we consider as a topological space with discrete topology,  
we take a disjoint union 
\[
\bigsqcup_{n \geq 0} [0,1]^n \times X^n 
\]
and consider the following relation $\sim$: 
\[\begin{array}{l}
(t_1, \ldots, t_{i-1}, 0, t_{i+1}, \ldots, t_n; 
x_1, \ldots, x_{i-1}, x_i, x_{i+1}, \ldots, x_n) \\
\qquad\qquad\sim 
(t_1, \ldots, t_{i-1}, t_{i+1}, \ldots, t_n; 
x_1, \ldots, x_{i-1}, x_{i+1}, \ldots, x_n), \\
(t_1, \ldots, t_{i-1}, 1, t_{i+1}, \ldots, t_n; 
x_1, \ldots, x_{i-1}, x_i, x_{i+1}, \ldots, x_n) \\
\qquad\qquad\sim 
(t_1, \ldots, t_{i-1}, t_{i+1}, \ldots, t_n; 
x_1 \us x_i, \ldots, x_{i-1} \us x_i, x_{i+1} \os x_i, \ldots, x_n \os x_i). 
\end{array}\]
Then the \textit{birack space} \cite{frs1, frs2}, denoted by $BX$, 
is defined to be the quotient space. 
We remark that the quotient map defines a structure of a CW complex on $BX$, where each $n$-cell is labeled by an $n$-tuple $(x_1,\dots,x_n) \in X^n$, and we find that $\pi_1(BX) \cong {\rm As}(X)$ from the 2-skeleton. Furthermore, we set another relation $\sim_D$: 
\[\begin{array}{l}
(t_1,\dots,t_{i-1}, t_i, t_{i+1}, t_{i+2}, \dots, t_n; x_1,\dots, x_{i-1}, x, x, x_{i+2}, \dots, x_n) \\
\qquad\qquad \sim_D
(t_1,\dots,t_{i-1}, t'_i, t'_{i+1}, t_{i+2}, \dots, t_n; x_1,\dots, x_{i-1}, x, x, x_{i+2}, \dots, x_n)
\end{array}\]
when $t_i+t_{i+1} = t'_i + t'_{i+1}$. 
We call the quotient space $\left(\bigsqcup_n [0,1]^n \times X^n\right)/(\sim,\sim_D)$ the \textit{biquandle space} and denote it by $B^QX$. 
We note that the biquandle (classifying) space might be first mentioned in \cite{f} 
with no rigorous definition. 
Let $X_D^n \subset X^n$ be the set
\[\{(x_1,\dots,x_n) \in X^n \mid \text{$x_i = x_{i+1}$ for some $i$}\}
\ \text{for}\ n \geq 2 \quad \text{and}\quad X_D^0 = X_D^1=\emptyset, 
\]
and define $X_{ND}^n = X^n \backslash X_D^n$. The composite
\[\bigsqcup_{n\geq0} [0,1]^n \times X_{ND}^n \hookrightarrow \bigsqcup_{n \geq 0} [0,1]^n \times X^n \to B^QX\]
defines a CW-complex structure on $B^QX$, and we see that $\pi_1(B^QX) \cong {\rm As}(X)$ from the 2-skeleton. By definition, we have the quotient map $BX \to B^QX$, and it maps an $n$-cell of $BX$ labeled by $X_{ND}^n$ to the cell of $B^QX$ labeled by the same element and one labeled by $X_D^n$ to the $(n-1)$-skeleton $B^QX^{(n-1)}$; this is a cellular map.

By equipping a quandle with a biquandle structure as in Remark~\ref{rem:qbq},
we can define the \textit{quandle space} as the biquandle space. 
We should note that quandle spaces have already been 
\lq\lq defined\rq\rq\ in \cite{nos2, cly}. 
See Remark~\ref{rem:q-space} below for the meaning of quotation marks.

\begin{remark}\label{rem:q-space}
Quandle spaces were introduced in \cite{nos1} up to $3$-skeletons, 
and the whole structures were defined in \cite[Definition 2.1]{nos2}. 
For a quandle $Q$, Nosaka \cite{nos2} considered a map from 
a disjoint union of infinitely many cells, which correspond to degenerate parts, 
to the (bi)rack space $B Q$. 
Then he defined a quandle space as the cone of his map.  
However his construction gives us an undesired space. 
Since each cell of the disjoint union is contractible, 
his quandle space is homotopy equivalent to the wedge sum of 
$B Q$ and a bouquet of infinitely many circles. 
For example, 
the fundamental group of his space is not isomorphic to $\mathrm{As}(Q)$, 
since it is 
the free product of $\mathrm{As}(Q)$ and the free group of infinite rank. 
It should be noted that 
the explanation after \cite[Definition 2.1]{nos2} 
describes a $4$-skeleton of a desired space  
by attaching additional cells to the (bi)rack space. 
%
This low-dimensional part accords with 
the corresponding part of our space above (when X is a quandle) 
up to homotopy equivalence.
%
We could also define the whole biquandle space by attaching cells 
to the birack space, though we omit it. 

There was an alternative definition \cite[p.49]{cly} of quandle spaces.  
Carter, Lebed and Yang \cite{cly} defined a quandle space by 
taking a quotient of the (bi)rack space with respect to 
a relation, which was essentially defined in \cite[Remark 2.5]{lv1}.  
However their relation is too strong and 
gives us an undesired space. 
Hereafter, in this remark, 
we rely on the reader's familiarity with \cite[Section 3]{cly}. 
They considered the relation defined by 
$( x, \widetilde{S}_j(p) ) \sim ( \widetilde{s}^j(x), p )$, 
where an essential part of $\widetilde{S}_j$ is 
a map from $[0,1]^2$ to $[0,1]$ 
sending $(s,t)$ to $s + t -st$. 
Since this map sends $(s,1)$ to $1$ for any $s \in [0,1]$ 
and also sends $(1,t)$ to $1$ for any $t \in [0,1]$, 
we can observe that their relation identifies all points of $( \widetilde{s}^j(x), p )$ 
whose $(j+1)$-st or $(j+2)$-nd entries equal to $1$. 
Hence their quandle space turns out to be a single point. 
It should be noted that if we slightly change the definition of $\widetilde{S}_j$ 
then their space with respect to the modified relation 
may accord with our space above (when $X$ is a quandle). 
The essential part of $\widetilde{S}_j$ must be changed into 
a map from $[0,1]^2$ to $[0,1]$ 
sending $(s,t)$ to the fractional part of $s + t$, 
though we omit details. 
\end{remark}

Let $X$ be a biquandle and $\col{C}$ an $X$-coloring on 
an oriented link diagram $D$. 
These data define a map $(\Xi_X(D))(\col{C}): S^2 \to BX$ as follows. 
First, we regard $D \subset S^2 =\mathbb{R}^2 \cup \{ \infty \}$ as a 4-valent graph and take a dual decomposition: a crossing point correspond to a 2-cell, a semi-arc to a 1-cell, and a region to a 0-cell (we here assume the closure of every region to be a disc for the simplicity). Then, we let $(\Xi_X(D))(\col{C})$ map every 0-cell to the single vertex of $BX$ and each 1-cell corresponding to $\sigma \in \col{SA}(D)$ to the 1-cell of $BX$ labeled by $\col{C}(\sigma)$. Finally, we define $(\Xi_X(D))(\col{C})$ so that it sends a 2-cell to the 2-cell labeled by $(\col{C}(\sigma_1), \col{C}(\sigma_2))$, where $\sigma_1$ and $\sigma_2$ are the two semi-arcs faced with the specified region of the crossing and $\sigma_1$ is the one belonging to the under arc. 
The map $(\Xi_X(D))(\col{C}): S^2 \to BX$ is defined up to homotopy and, by an abuse of notation, we denote the composite $S^2 \to BX \to B^QX$ with the quotient by the same symbol; we obtain an element $(\Xi_X(D))(\col{C}) \in \pi_2(B^Q X)$.
The map $\Xi_X(D) \colon \Col{X}{D} \to \pi_2(B^QX)$ is called a 
\textit{biquandle homotopy invariant} of $L$. 
When $X$ is finite,  $\Xi_X(D)$ can be considered as an element
\[
\sum_{\mathcal{C} \in \Col{X}{D}} ( \Xi_X(D) )( \mathcal{C} ) 
\]
in the group ring $\Z [\pi_2(B^QX)]$ and 
does not depend on the choice of the diagram $D$. 
We should note that Nosaka defined the \textit{quandle homotopy invariant} 
for links in \cite{nos1} (and for surface links in \cite{nos2}). 
By equipping a quandle with a biquandle structure as in Remark~\ref{rem:qbq}, 
we can also define the quandle homotopy invariant as 
the biquandle homotopy invariant.

\subsection{Generalization of one-to-one 
correspondence}\label{subsec:one-to-one}
For a biquandle $X$, a map $\psi_{\mathrm{As}} \colon \mathrm{As}(\Qdl{X}) \to \mathrm{As}(X)$ 
is defined (inductively) by 
\[\begin{array}{rcl}
\psi_{\mathrm{As}}(e) & := & e, \\
\psi_{\mathrm{As}}(p \cdot a) & := & \psi_{\mathrm{As}}(p) \cdot \psi(p, a), \\
\psi_{\mathrm{As}}(p \cdot a^{-1}) & := & \psi_{\mathrm{As}}(p) \cdot \psi(p \cdot a^{-1}, a)^{-1}
\end{array}\] 
for $p \in \mathrm{As}(\Qdl{X})$ and $a \in \Qdl{X}$. We can check that $\psi_{\mathrm{As}}$ is well defined; e.g.,
\begin{align*}
\psi_{\mathrm{As}}(p \cdot b \cdot (a \s b)) &= \psi_{\mathrm{As}}(p)\cdot \psi(p,b) \cdot \psi(p\cdot b, a \s b)\\
&= \psi_{\mathrm{As}}(p) \cdot \psi(p, b) \cdot (\psi(p, a) \us \psi(p, b))\\
&= \psi_{\mathrm{As}}(p) \cdot \psi(p, a) \cdot (\psi(p, b) \os \psi(p, a))\\
&= \psi_{\mathrm{As}}(p) \cdot \psi(p, a) \cdot \psi(p\cdot a,b)\\
&= \psi_{\mathrm{As}}(p \cdot a \cdot b)
\end{align*}
for $p \in \mathrm{As}(\Qdl{X})$ and $a, b \in \Qdl{X}$, where the second equality follows from the equation (\ref{eq:us}) in the proof of Lemma~\ref{lem:psi}. 
Let $\tilde{\psi} \colon \mathrm{As}(\Qdl{X}) \times \Qdl{X} \to \mathrm{As}(X) \times X$ 
be the map defined by $\tilde{\psi}(p,a) = ( \psi_{\mathrm{As}}(p), \psi(p,a) )$ 
for $p \in \mathrm{As}(\Qdl{X})$ and $a \in \Qdl{X}$. 

\begin{lem}\label{lem:psi-as-bij}
The maps $\psi_{\mathrm{As}}: \mathrm{As}(\Qdl{X}) \to \mathrm{As}(X)$ and $\tilde{\psi} \colon \mathrm{As}(\Qdl{X}) \times \Qdl{X} \to \mathrm{As}(X) \times X$ are bijective.
\end{lem}
\begin{proof}
We inductively define a map $\phi_{\mathrm{As}} \colon \mathrm{As}(X) \to \mathrm{As}(\Qdl{X})$ by 
\[\begin{array}{rcl}
\phi_{\mathrm{As}}(e) & := & e, \\
\phi_{\mathrm{As}}(g \cdot x) & := & \phi_{\mathrm{As}}(g) \cdot \phi(g, x), \\
\phi_{\mathrm{As}}(g \cdot x^{-1}) & := & \phi_{\mathrm{As}}(g) \cdot \phi(g \cdot x^{-1}, x)^{-1}
\end{array}\]
for $g \in \mathrm{As}(X)$ and $x \in X$. As in the case of $\psi_{\mathrm{As}}$, we can show that $\phi_{\mathrm{As}}$ is well defined. 
Let $\tilde{\phi} \colon \mathrm{As}(X) \times X \to \mathrm{As}(\Qdl{X}) \times \Qdl{X}$ be the map defined by $\tilde{\phi}(g,x) = ( \phi_{\mathrm{As}}(g), \phi(g,x) )$ for $g \in \mathrm{As}(X)$ and $x \in X$.

Let us simultaneously show that $\phi_{\mathrm{As}}$ is a left inverse of $\psi_{\mathrm{As}}$ and that $\tilde{\phi}$ is a left inverse of $\tilde{\psi}$ by an induction. First, we find from the definitions that $\phi_{\mathrm{As}} \circ \psi_{\mathrm{As}}(e) = e$ and $\tilde{\phi} \circ \tilde{\psi}|_{\{e\} \times \Qdl{X}} = \text{id}_{\{e\} \times \Qdl{X}}$. Next, under the assumption that $\phi_{\mathrm{As}} \circ \psi_{\mathrm{As}}(p) = p$ and $\tilde{\phi} \circ \tilde{\psi}|_{\{p\} \times \Qdl{X}} = \text{id}_{\{p\} \times \Qdl{X}}$ for certain $p \in \mathrm{As}(\Qdl{X})$, we have
\begin{align*}
\tilde{\phi} \circ \tilde{\psi}(p\cdot b, a) &= \tilde{\phi}(\psi_{\mathrm{As}}(p) \cdot \psi(p,b),\psi(p,a) \os \psi(p,b))\\
&= (\phi_{\mathrm{As}}\circ\psi_{\mathrm{As}}(p) \cdot \phi(\psi_{\mathrm{As}}(p),\psi(p,b)),\phi(\psi_{\mathrm{As}}(p) \cdot \psi(p,b),\psi(p,a) \os \psi(p,b)))\\
&= (p \cdot b, \phi(\psi_{\mathrm{As}}(p),\psi(p,a)))\\
&= (p \cdot b, a)
\end{align*}
for any $a,b \in \Qdl{X}$, where the third and the fourth equalities follows from the assumption. 
This means that $\tilde{\phi} \circ \tilde{\psi}|_{\{p \cdot b\} \times \Qdl{X}} = \text{id}_{\{p \cdot b\} \times \Qdl{X}}$ and hence $\phi_{\mathrm{As}} \circ \psi_{\mathrm{As}}(p\cdot b) = p \cdot b$. In a similar way, we can check that $\tilde{\phi} \circ \tilde{\psi}|_{\{p \cdot b^{-1}\} \times \Qdl{X}} = \text{id}_{\{p \cdot b^{-1}\} \times \Qdl{X}}$ and $\phi_{\mathrm{As}} \circ \psi_{\mathrm{As}}(p\cdot b^{-1}) = p \cdot b^{-1}$. Thus, we find that $\phi_{\mathrm{As}} \circ \psi_{\mathrm{As}} = \text{id}$ and $\tilde{\phi} \circ \tilde{\psi} = \text{id}$.
\par We also find $\psi_{\mathrm{As}} \circ \phi_{\mathrm{As}}$ and $\tilde{\psi} \circ \tilde{\phi}$ to be the identities, similarly.
\end{proof}

\begin{prop}\label{prop:local-bij}
The bijection between the shadow colorings induced by $\Psi$ is equal to $\tilde{\psi}$, i.e., for any $\col{C} \in \Col{\Qdl{X}}{D}$ we have
\[ {( \Psi(\col{C}) )}^S = \tilde{\psi}\circ\col{C}^S \qquad \in {\rm Map}(\col{SA}(D), \mathrm{As}(X) \times X).\]
\end{prop}
\begin{proof}
By Lemma \ref{lem:psi=psi}, we have $\Psi(\col{C}) = \psi\circ\col{C}^S$. Furthermore, the definition of $\psi_{\mathrm{As}}$ implies that $\psi_{\mathrm{As}} \circ \col{C}^R$ defines a region coloring of $\psi \circ \col{C}^S$: the map $\psi_{\mathrm{As}} \circ \col{C}^R$ satisfies the second condition of the definition of the region coloring.
Since $\psi_{\mathrm{As}} \circ \col{C}^R(R_\infty) = e$ 
for the unbounded region $R_\infty$ of $D$, the map $\col{C}^R$ is the region coloring, 
which implies the required equation ${( \Psi(\col{C}) )}^S = \tilde{\psi}\circ\col{C}^S$.
\end{proof}

\subsection{Equivalence of two homotopy invariants}\label{subsec:equiv}
Let $X$ be a biquandle. We denote the universal covering of $BX$ by $\widetilde{BX}$ and that of $B^Q X$ by $\widetilde{B^Q X}$. We remark that $\widetilde{BX}$ (resp. $\widetilde{B^Q X}$) has a CW-complex structure induced by the covering map $\widetilde{BX} \to BX$ (resp. $\widetilde{B^QX} \to B^Q X$), where each $n$-cell is labeled by $\mathrm{As}(X) \times X^n$ (resp. $\mathrm{As}(X) \times X^n_{ND}$) since $\pi_1(BX) \cong \pi_1(B^Q X) \cong \mathrm{As}(X)$.

\begin{prop}
We have isomorphisms
\[ \psi_B: \widetilde{B\Qdl{X}} \to \widetilde{BX} \quad \text{and} \quad \psi^Q_B: \widetilde{B^Q \Qdl{X}} \to \widetilde{B^Q X} \]
of CW complexes.
\end{prop}
\begin{proof}
We define a map $\psi_{B}^0: \bigsqcup_{n \geq 0} [0,1]^n \times \mathrm{As}(\Qdl{X}) \times \Qdl{X}^n \to \bigsqcup_{n \geq 0} [0,1]^n \times \mathrm{As}(X) \times X^n$ as
\[ \psi_{B}^0: (t; p, a_1, \dots, a_n) = (t; \psi_{\mathrm{As}}(p), \psi(p,a_1), \dots, \psi(p,a_n)) \]
for $t \in [0,1]^n, p \in \mathrm{As}(\Qdl{X}),$ and $a_1,\dots,a_n \in \Qdl{X}$. By Lemma \ref{lem:psi-as-bij}, this is a homeomorphism. We easily find $\psi_{B}^0$ compatible with the characteristic maps defining the CW-complex structures of $\widetilde{B\Qdl{X}}$ and $\widetilde{BX}$, and hence a map $\psi_B: \widetilde{B\Qdl{X}} \to \widetilde{BX}$ is uniquely defined so that a diagram
\[ \begin{CD}
\bigsqcup_{n \geq 0} [0,1]^n \times \mathrm{As}(\Qdl{X}) \times \Qdl{X}^n @>\psi_{B}^0 >> \bigsqcup_{n \geq 0} [0,1]^n \times \mathrm{As}(X) \times X^n \\
@VVV @VVV \\
\widetilde{B\Qdl{X}} @>\psi_B >> \widetilde{BX}
\end{CD} \]
commutes. Since $\psi_{B}^0$ is a homeomorphism, $\psi_B$ is an isomorphism of CW complexes. 
\par A proof for biquandle spaces is similar: we can check the compatibility of $\psi_B^0$ with a relation $\sim_{\tilde{D}}$, where $\sim_{\tilde{D}}$ is the relation of $[0,1]^n \times \mathrm{As}(X) \times X^n$ given by
\[(t;g,x) \sim_{\tilde{D}} (t';g',x') \quad \Leftrightarrow \quad (t;x) \sim_D (t',x') \;\text{and}\;g=g'\]
for $t, t'\in [0,1]^n, g,g' \in \mathrm{As}(X), x,x' \in X^n$ ($\sim_{\tilde{D}}$ is defined also for $[0,1]^n \times \mathrm{As}(\Qdl{X}) \times \Qdl{X}^n$), and then $\psi^Q_B: \widetilde{B^Q \Qdl{X}} \to \widetilde{B^Q X}$ is defined and is an isomorphism.
\end{proof}

We define an isomorphism $\psi_*$ as the composite 
\[ \pi_n(B^Q\Qdl{X}) \xrightarrow{(p_{\Qdl{X} *})^{-1}} \pi_n(\widetilde{B^Q\Qdl{X}}) \xrightarrow{\psi^Q_{B*}} \pi_n(\widetilde{B^QX}) \xrightarrow{p_{X*}} \pi_n(B^QX) \]
for $n \geq 2$, where $p_{\Qdl{X} }: \widetilde{B^Q \Qdl{X}} \to B^Q \Qdl{X}$ and $p_X: \widetilde{B^Q X} \to B^Q X$ are the covering maps. We use the same notation $\psi_*$ also for the induced isomorphism of the group rings. The following theorem means that the two homotopy invariants with respect to $\Qdl{X}$ and $X$ are equivalent under the identification of the homotopy groups via $\psi_*$.

\begin{theorem}\label{thm:bhi}
Let $X$ be a biquandle and $D$ a diagram of an oriented link. For any quandle coloring $\col{C} \in \Col{\Qdl{X}}{D}$, we have
\[ \psi_*((\Xi_{\Qdl{X}}(D))(\col{C})) = (\Xi_X(D))(\Psi(\col{C})) \quad \in \pi_2(B^Q X). \]
In particular, if $X$ is finite,
\[ \psi_* (\Xi_{\Qdl{X}}(D)) = \Xi_X(D) \quad \in \mathbb{Z}[\pi_2(B^Q X)]. \]
\end{theorem}
\begin{proof}
Let $\tilde{\Xi}_{\Qdl{X}}(\col{C}): S^2 \to \widetilde{B^Q\Qdl{X}}$ be the lift of $\Xi_{\Qdl{X}}(\col{C}): S^2 \to B^Q\Qdl{X}$
which maps $\infty \in S^2$ to the 0-cell labeled by the identity element $e \in {\rm As}(\Qdl{X})$,
where we omit ``$D$'' to simplify the notation. Here we should notice that by $\tilde{\Xi}_{\Qdl{X}}(\col{C})$ every 1-cell corresponding to $\sigma \in \col{SA}(D)$ is mapped to the 1-cell labeled by $\col{C}^S(\sigma)$ of $\widetilde{B^Q\Qdl{X}}$. 
Then $\psi^Q_{B}$ sends it to the one labeled by $\tilde{\psi} \circ \col{C}^S(\sigma)$, which is equal to ${( \Psi(\col{C}) )}^S(\sigma)$ by Proposition \ref{prop:local-bij}. Finally, the quotient map $p_X$ maps it to the one labeled by $(\Psi(\col{C}))(\sigma)$. The analogous things also hold for the 0- and 2-cells, and this implies the first equation: $p_{X*}\circ \psi^Q_{B*} \circ \tilde{\Xi}_{\Qdl{X}}(\col{C}) = \Xi_X(\Psi(\col{C}))$. The second follows from the first and Theorem \ref{thm:main} immediately.
\end{proof}

\begin{remark}
Let $X$ be a biquandle and $A$ an abelian group. 
For a biquandle $2$-cocycle $\theta \colon X \times X \to A$, 
let $\langle \bullet , \theta \rangle \colon H_2(B^Q X; \Z) \to A$ be the contraction by $\theta$ 
(through the Kronecker product). 
Then 
the biquandle cocycle invariant $\Phi_\theta(D) \in \Z[A]$ is derived from 
the biquandle homotopy invariant $\Xi_X(D) \in \Z [\pi_2(B^QX)]$ by 
\[
\Phi_\theta(D) = \langle \bullet , \theta \rangle \circ \mathfrak{H} (\Xi_X(D)) , 
\]
where a map $\mathfrak{H}_X$ is the Hurewicz homomorphism 
\[
\mathfrak{H}_X \colon \pi_2(B^QX) \to  H_2(B^Q X; \Z)
\] 
and the two maps in the above equation are naturally extended to group rings. 
We should note that Nosaka \cite{nos1} gave this interpretation 
when $X$ is a quandle. 

Let $Q$ be a quandle and $A$ an abelian group. 
For a shadow quandle $2$-cocycle 
$\theta \colon \mathrm{As}(Q) \times Q \times Q \to A$, 
let $\langle \bullet , \theta \rangle \colon H_2(\widetilde{B^Q Q}; \Z) \to A$ be the contraction by $\theta$ 
(through the Kronecker product). 
Then 
the shadow quandle cocycle invariant $\Phi_\theta^{S}(D) \in \Z[A]$ is derived from 
the quandle homotopy invariant $\Xi_Q (D) \in \Z [\pi_2(B^Q Q)]$ by 
\[
\Phi_\theta(D) = \langle \bullet , \theta \rangle \circ \widetilde{\mathfrak{H}}_Q (\Xi_Q(D)) , 
\]
where a map $\widetilde{\mathfrak{H}}_Q$ is the Hurewicz homomorphism (with local coefficients) 
\[
\widetilde{\mathfrak{H}}_Q \colon \pi_2(B^Q Q) \xrightarrow[\cong]{(p_{Q*})^{-1}} \pi_2(\widetilde{B^Q Q})\to  H_2(\widetilde{B^Q Q}; \Z)
\] 
and the two maps in the above equation are naturally extended to group rings.  
We should note that Nosaka \cite{nos2} gave this interpretation for surface links. 
\end{remark}

\begin{proof}[An alternative proof of Theorem~\ref{thm:bci}]
It holds that a diagram 
\[\begin{CD}
\pi_2(B^Q \mathcal{Q}(X)) @>{\widetilde{\mathfrak{H}}_{\mathcal{Q}(X)}}>> H_2(\widetilde{B^Q \mathcal{Q}(X)}; \Z) @>{\langle \bullet , \psi^* \theta \rangle}>> A \\
@V{\psi_*}V{\cong}V @V{ (p_X \circ \psi^Q_B)_* }VV @| \\
\pi_2(B^Q X) @>{\mathfrak{H}_X}>> H_2(B^Q X; \Z) @>\langle \bullet , \theta \rangle>> A 
\end{CD}\]
commutes. Hence we have 
\begin{align*}
\Phi_\theta(D) 
& = \langle \bullet , \theta \rangle \circ \mathfrak{H}_X (\Xi_X(D)) \\
& = \langle \bullet , \theta \rangle \circ \mathfrak{H}_X ( \psi_* (\Xi_{\Qdl{X}}(D)) ) \\
& = \langle \bullet , \psi^* \theta \rangle 
\circ \widetilde{\mathfrak{H}}_{\mathcal{Q}(X)} (\Xi_{\mathcal{Q}(X)}(D)) \\
& = \Phi_{\psi^* \theta}^{S} (D) ,
\end{align*}
where the second equality follows from Theorem~\ref{thm:bhi} and 
the third equality follows from the above commutativity. 
\end{proof}

\section{For surface links}\label{sec:2-dim}
A \textit{surface link} $L$ is a closed surface smoothly embedded in $\mathbb{R}^4$. In this paper, we always assume a surface link to be oriented. We take a generic projection 
from $\R^4$ to $\R^3$, 
and we call the image of $L$ with the information of the relative height, 
a (\textit{broken surface}) \textit{diagram}. 
Here the height information is usually 
indicated by removing the regular neighborhood of a double point in the lower sheets
along the double point curves. 
Using diagrams, we can define the following invariants of surface links as well as 
those of classical links: 
quandle (shadow) colorings \cite{cjkls, cks, cegs}, 
biquandle (shadow) colorings \cite{kkkl}, 
fundamental (bi)quandles \cite{car,ash}, 
quandle cocycle invariants \cite{cjkls}, 
shadow quandle cocycle invariants \cite{cks, cegs}, 
biquandle cocycle invariants \cite{kkkl}, 
quandle homotopy invariants \cite{nos2},  
biquandle homotopy invariants, 
and so on. 
We note that a rigorous definition of biquandle homotopy invariants 
do not exist in any literature, 
but it could be easily written down explicitly if  
we mimic the definition of quandle homotopy invariants 
by replacing quandle spaces in the definition with biquandle spaces. 
%
For more details on studies of surface links using quandle theory, 
we refer to \cite{cks-book, kam}. 
\par The theorems for classical links in this paper are similarly shown with a little modification; e.g., taking a diagram without a branch point \cite{cs, hir} makes it easier to construct the one-to-one correspondence between biquandle colorings and quandle colorings. Here, we just state the results:

\begin{theorem}\label{thm:main2}
Let $X$ be a biquandle and $D$ be a diagram of an oriented surface link. 
Then there exists a one-to-one correspondence 
between $\Col{X}{D}$ and $\Col{\Qdl{X}}{D}$. 
\end{theorem}

\begin{theorem}
For a biquandle $X$ and an abelian group $A$, 
let $\theta \colon X \times X \times X \to A$ be a biquandle $3$-cocycle. 
Then, for a diagram $D$ of an oriented surface link, 
the biquandle cocycle invariant with respect to $\theta$ 
can be considered as the shadow quandle cocycle invariant with respect to 
$\psi^* \theta$. 
In particular, we have 
\[\Phi_{\theta}(D) = \Phi_{\psi^* \theta}^{S} (D)\] 
in $\Z[A]$ if $X$ is finite.
\end{theorem}

\begin{theorem}
Let $X$ be a biquandle and $D$ a diagram of an oriented surface link. For any quandle coloring $\col{C} \in \Col{\Qdl{X}}{D}$, we have
\[ \psi_*((\Xi_{\Qdl{X}}(D))(\col{C})) = (\Xi_X(D))(\Psi(\col{C})) \quad \in \pi_3(B^Q X). \]
In particular, if $X$ is finite,
\[ \psi_* (\Xi_{\Qdl{X}}(D)) = \Xi_X(D) \quad \in \mathbb{Z}[\pi_3(B^Q X)]. \]
\end{theorem}




\section*{Acknowledgments}
The authors thank Seung Yeop Yang for letting us know the reference \cite{f} of biquandle spaces.  
The first-named author has been supported in part by 
Grant-in-Aid for JSPS Fellows, (No.~JP16J01183), 
Japan Society for the Promotion of Science. 
The second-named author has been supported in part by 
the Grant-in-Aid for Scientific Research (C), (No.~JP17K05242), 
Japan Society for the Promotion of Science.


\appendix
\section{The topological biquandle of an oriented link}\label{sec:app}

In this appendix, we point out that the topological biquandle is not isomorphic to 
the fundamental biquandle for the trivial knot. 
The \textit{topological biquandle} $\widehat{\mathcal{B}}_L$ of an oriented link $L$ 
is introduced in \cite{hor}. 
As a set, it is defined as the quotient of the set of pairs of two paths
$$\{(a_0,a_1) \mid a_0,a_1: [0,1] \to E_L, a_0(0) = a_1(0) \in \partial E_L, a_0(1) = p, a_1(1) =q\}$$
by a suitable homotopy equivalence, 
where $E_L$ is (the closure of) an exterior of $L$ and $p,q \in E_L$ are two fixed points. 
The two binary operations on $\widehat{\mathcal{B}}_L$ are topologically defined 
in a way similar to the topological definition of the fundamental quandle $Q_L$ of $L$. 
We can check that $\widehat{\mathcal{B}}_L$ is isomorphic to a biquandle 
$\mathrm{As}(Q_L) \times Q_L$ with two binary operations defined by 
$$ 
(g,a) \us (h,b) := (g \cdot b, a \s b) \quad \text{and} \quad 
(g,a) \os (h,b) := ( (h \cdot b \cdot h^{-1}) \cdot g, a)
$$
for $(g,a), (h,b) \in \mathrm{As}(Q_L) \times Q_L$
through an isomorphism sending (the equivalence class of) 
an element $(a_0,a_1)$ in $\widehat{\mathcal{B}}_L$ to 
$(l_\infty  a_1^{-1}  a_0, a_0)$ in $\mathrm{As}(Q_L) \times Q_L$,
where $l_\infty$ is a fixed path from $p$ to $q$ in the exterior of a $3$-ball 
whose interior contains the whole of $L$. 
We should note 
that $\mathrm{As}(Q_L)$ is naturally isomorphic to $\pi_1(\R^3 \setminus L, p)$, 
and that the definition of biquandles in \cite{hor} is different from ours 
but they are equivalent. 
She uses Kauffman's corner operations \cite{fjk}, 
the up operation $\up{\vbox to 1.1ex {\, }}$ and the down operation $\dow{\vbox to 1.1ex {\, }}$, 
which can be written as 
$$x \up{y} = x \us (y \os^{\, -1} x) \quad \text{and} \quad x \dow{y} = x \os^{\, -1} y$$
for elements $x,y$ in a given biquandle 
in terms of our binary operations $\os$ and $\us$.

For example, 
the topological biquandle $\widehat{\mathcal{B}}_K$ of the trivial knot $K$ 
is isomorphic to 
a biquandle $\Z$ with two binary operations 
defined by 
$$x \us y := x+1 \quad \text{and} \quad  x\os y := x+1$$ 
for $x,y \in \mathbb{Z}$, 
since $Q_K$ is isomorphic to the trivial quandle of one element and 
$\mathrm{As}(Q_K)$ is isomorphic to the infinite cyclic group. 
On the other hand, 
the fundamental biquandle of the trivial knot $K$ is isomorphic to 
the free biquandle $F$ of one generator. 
Although the right multiplications $\us x$ (and also $\os x$) are independent 
of the choice of $x$ in the biquandle $\Z$, 
we see that this is not the case in the free biquandle $F$ as follows. 
Since a biquandle $\Z^2$ with two binary operations, defined by 
$$(x,a) \us (y,b) := (x+1, a+y) \quad \text{and} \quad 
(x,a) \os (y,b) := (x+1, a+y)$$ 
for $(x,a),(y,b) \in \Z^2$, can be generated by one element $(0,0)$, 
we have a surjective homomorphism $F \to \Z^2$ 
and it holds that $\us (x,a) \neq \us (y,b)$ if $x \neq y$ in its image. 
This implies that the biquandle $\Z$ is not isomorphic to the free biquandle $F$, 
and hence we conclude that the topological biquandle $\widehat{\mathcal{B}}_K$ 
is not isomorphic to the fundamental biquandle of the trivial knot $K$.




\end{document}